\documentclass{amsart}
\usepackage{amsmath}%
\usepackage{amsfonts}%
\usepackage{amssymb}%
\usepackage{graphicx}
\usepackage{float}
\usepackage{bbold}
\usepackage{caption}
\usepackage{subcaption}
\usepackage{xcolor} 

\newtheorem{theorem}{Theorem}
\theoremstyle{plain}

\newtheorem{corollary}{Corollary}

\newtheorem{definition}{Definition}

\newtheorem{lemma}{Lemma}

\newtheorem{proposition}{Proposition}
\newtheorem{remark}{Remark}

\newtheorem*{note}{Note}

\numberwithin{equation}{section}
\begin{document}
\nocite{*}
\raggedbottom
\title{Braidoids}
\author{Nesl\.{i}han G\"ug\"umc\"u}
\address{ Izmir Institute of Technology 
Department of Mathematics 
G\" ulbah\c ce Mah. URLA 35430 
\.Izmir, TURKEY}
\email{neslihangugumcu@iyte.edu.tr}%
\author{Sofia Lambropoulou}
\address{School of Applied Mathematical and Physical Sciences
National Technical University of Athens
Zografou Campus
GR-15780 Athens, GREECE}
\email{sofia@math.ntua.gr}%
\urladdr{http://www.math.ntua.gr/$\sim$sofia/}


\begin{abstract}
Braidoids generalize the classical braids and form a counterpart theory to the theory of planar knotoids, just as the theory of braids does for the theory of knots. In this paper, we introduce basic notions of braidoids, a closure operation for braidoids,  we prove an analogue of the Alexander theorem, that is, an algorithm that turns a knotoid into a braidoid, and we formulate and prove a geometric analogue of the Markov theorem for braidoids using the $L$-moves.
\end{abstract}

\maketitle

\section{Introduction}
A knotoid in an oriented surface $\Sigma$ is an equivalence class of oriented open-ended knot diagrams in $\Sigma$, with two endpoints that can lie in any local region determined by the diagram. The equivalence is generated by the Reidemeister moves and isotopies of $\Sigma$, which include the swinging of an endpoint within a region free of endpoints. View Figures \ref{fig:knotoid} and \ref{fig:om}. When $\Sigma$ is, in particular,  the 2-sphere $S^2$ the knotoids are named spherical knotoids. When $\Sigma$ is the plane they are named planar knotoids. 

The theory of knotoids was introduced by V. Turaev in 2011 \cite{Tu}. Turaev showed that the set of classical knots injects into the set of spherical knotoids, where a knot can be viewed as a knotoid with zero complexity. The complexity of a knotoid $K$ (or height, as the term used in \cite{GK1}) is the minimum number  of crossings over all diagrams of $K$, that are created when realizing the end-to-end (underpass) closure of each diagram of $K$ to a knot diagram. The above observation provides a strong motivation for using knotoids in computing classical knot invariants, especially those based on the number of crossings, since the computations would reduce exponentially with reducing the number of crossings in a diagram. We recall that the complete classification of knots is still a big open problem in Mathematics and it is tackled by constructing isotopy invariants for distinguishing pairs of different knots. Indeed, Turaev -among other results- defined the knot group in terms of knotoids, invariants for knotoids, such as the bracket polynomial, and he remarked on the fruitful connections of the theory with virtual knot theory. Turaev also showed that there is a surjective map from the set of planar knotoids to the set of spherical knotoids, which is not injective.

After the works \cite{Tu, DST} and \cite{Bar}, the interest in knotoids was rekindled by the second author, who proposed the subject to the first author for her PhD study \cite{Gthesis}. Then, in \cite{GK1} the first author with L.~H. Kauffman constructed new invariants for knotoids, considering the virtual end-to-end closure and employing techniques from virtual knot theory \cite{Kau2}. We recall that a virtual knot diagram contains classical as well as virtual crossings, where a virtual crossing has no information under or over, it roughly indicates a permutation of the two acs involved. Another interesting result in \cite{GK1} is the realization  of a  knotoid via its lifting to an open-ended curve embedded in 3-space. Through this lifting knotoids could serve as mathematical models for  proteins and the theory of knotoids could be used for analyzing their topology. Consequently, in \cite{GDBS} invariants of  spherical knotoids were applied, which rendered as much information as the use of the virtual closure, while in \cite{GGLDSK} the application of planar knotoids revealed richer structure.  

In parallel, the theory of braidoids was initiated by the authors of the present paper \cite{Gthesis,GL1} for counterparting the theory of planar knotoids, just as classical braids comprise  an algebraic counterpart to classical knots. A braid is a set of descending strands with paired top-to-bottom ends, and equivalence classes of braids under obvious planar isotopy and level preserving Reidemeister moves realize groups, the Artin braid groups of type $A$ \cite{Ar1,Ar2,KT}. The paired ends of a braid can be joined to form a closure, which is a knot or a link. The inverse operation consists in turning an oriented knot or link into an isotopic closed braid and this is always possible by the classical Alexander theorem \cite{Br,Al,Bir}. The Alexander theorem and its proof play a key role in the proof of the Markov theorem, which provides an equivalence relation among elements of the braid groups that corresponds precisely to the isotopy relation among oriented knots and links \cite{Ma,Wei,Bir,Ben}. The two theorems allow, in principle, for the use of braids and algebraic techniques in the study of knots and links and they were successfully used for the first time by V.~F.~R.~Jones  in his construction of the famous Jones polynomial \cite{jo1,jo2}. Consequently, these theorems  received anew the attention of several mathematicians \cite{Mo,Vo,Ya,Tr,Lathesis,LR1,BM}, whose works revealed diverse and interesting approaches.

In analogy to a braid diagram, a braidoid diagram consists in a set of descending strands such that two of them are special: one of them terminates at an endpoint, the head, and the other starts from the second endpoint, the leg. Either endpoint may lie in any region and at any height of the diagram. See Figure \ref{fig:basicbrds}  for examples. A braidoid, then, is the equivalence class of braidoid diagrams under obvious isotopy moves analogous to the knotoid equivalence moves, as  illustrated in Figures~\ref{fig:d}, \ref{fig:vertmo}, \ref{fig:brswing2}. In particular, endpoints may swing and there are also the forbidden moves in the theory, Figure~\ref{fig:f}. 
 Hence, the notion of braidoid extends the notion of classical braid. We study braidoids in relation to  knotoids by defining an appropriate closure operation, which respects the forbidden moves. The closure is defined on  braidoid diagrams with labelled ends, specifying whether the joining arcs will run all over or all under the rest of the diagram.  

An analogue of the classical Alexander theorem, stating that any knotoid diagram can be turned into a labeled braidoid diagram with isotopic closure,  was proved in \cite{GL1}. In this paper, we give a second proof of the Alexander theorem for braidoids, which is more adapted for proving an analogue of the classical Markov theorem. Indeed, we formulate and prove a geometric analogue of the Markov theorem for braidoids, using the concept of the $L$-move as introduced in \cite{Lathesis,LR1}. An $L$-move is a geometric as well as algebraic move for braids and it consists in cutting a braid arc at any point and pulling the two ends along the vertical line of the cutpoint, both over or both under the rest of the diagram, so as to obtain in the end a new pair of corresonding braid strands. The $L$-moves provide an one-move version of the classical Markov theorem \cite{LR1}. The set of braidoids does not support an obvious algebraic structure. So, it would not be possible to formulate a braidoid equivalence without the use of the $L$-moves. In \cite{Gthesis,GL1} a set of building blocks for braidoids is listed along with some relations that they satisfy. We then propose in  \cite{GL1} to encode a protein by the monomials of the building blocks for the braidoids corresponding to the  knotoids related to the protein. We finally relate braidoids to classical or virtual braids by defining appropriately the underpass resp. virtual closure. We recall that a virtual braid contains classical as well as virtual crossings, see \cite{Kau2,KL1}.

For further works on knotoids and applications the interested reader may consult \cite{AHKS,BBHL,DTRGDSSSMR,GN,KL,KMT}.

Let us now present the organization of the paper. In Section \ref{sec:knotoids} we review basics on knotoids. In Section \ref{sec:braidoids} we define the notion of braidoid diagram and the notion of braidoid by introducing  isotopy moves on braidoid diagrams. In Section \ref{sec:closure}, we explain a way to close a braidoid diagram with labels which will relate them to knotoids. Later, in Section \ref{sec:braiding} we describe an algorithm for turning a knotoid diagram into a labeled braidoid diagram with isotopic closure. This yields an analogue of the classical Alexander theorem for braidoids. In Section \ref{sec:L-eqv.} we adapt the classical $L$-moves, which were originally defined for braid diagrams by the second listed author in \cite{La1}, for braidoid diagrams. We also introduce the fake swing moves, which along with the $L$-moves comprise the $L$-equivalence. We then prove our geometric analogue of the Markov theorem for braidoids. Finally, in Section~\ref{sec:tobraids} we present  the underpass and virtual closures that relate a braidoid to a classical resp. a virtual braid.


\section{A review of knotoids}{\label{sec:knotoids}

Let $[0,1]$ be the unit interval and $\Sigma$ be any oriented surface. \textit{A knotoid diagram} $K$ in $\Sigma$ is an immersion $K: [0,1] \rightarrow \Sigma$ that is generic in the sense that there is only a finite number of double points appearing as transversal \textit{crossings} each with the extra information of under or over. The images of $0$ and $1$ are two distinct points disjoint from any crossings of $K$, too, and are called {\it leg} and {\it head} of $K$, respectively. Furthermore, $K$ inherits a natural orientation from its leg to its head. In Figure \ref{fig:knotoid} we show some examples of knotoid diagrams including the trivial knotoid diagram that admits no crossing.
 \begin{figure}[H]
\centering \includegraphics[width=1\textwidth]{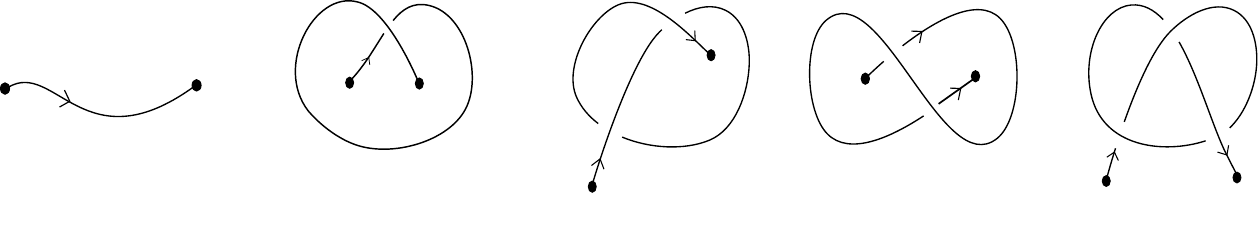}
\caption{Knotoid diagrams}
\label{fig:knotoid}
\end{figure}
 \begin{definition}\normalfont A \textit{piecewise-linear knotoid diagram} is a union of finitely many edges: $[p_1, p_2],...,[p_{n-1}, p_n]$ such that each edge intersects one or two other edges at the vertices, $p_i$, for $i=2,...,n-1$. The vertices $p_1$ and $p_n$ correspond to the endpoints of the diagram. Two edges can also intersect transversely at double points endowed with over/under-data, called crossings of the diagram.\end{definition} %

Any classical knot diagram can be turned into a piecewise-linear knot diagram and piecewise-linear isotopy classes are in bijection with ambient isotopy classes of knots \cite{BZ}. In a similar way any knotoid diagram can be turned into a piecewise-linear knotoid diagram. 
In this paper, we will be working for convenience with piecewise-linear knotoid diagrams. Moreover we will be considering $\Sigma=\mathbb{R}^2$.

\subsection{Moves on knotoid diagrams}\label{sec:knomo}

 A $\Delta$-\textit{move} on a knotoid diagram is a replacement of an arc with two arcs (or vice versa) forming a triangle which does not contain any of the endpoints of the knotoid diagram and passing entirely over or under any arcs intersecting this triangle, as shown in Figure \ref{fig:om}. The Reidemeister moves $\Omega_1$, $\Omega_2$, $\Omega_3$ (Figure \ref{fig:om}) are some special cases of $\Delta$-moves. We shall call a $\Delta$-move that takes place in a triangular region that does not contain any arcs of the diagram in its interior, an $\Omega_0$-move.  Finally, we have the \textit{swing moves} whereby an arc containing an endpoint sweeps a triangular region free of any other arcs of the diagram. The swing moves can be viewed as special cases of $\Omega_0$-moves where one side of the isotopy triangle is omitted. See Figure \ref{fig:om}. The moves consisting of pulling the strand adjacent to an endpoint over or under a transversal strand as shown in Figure \ref{subfig:for} are the \textit{forbidden knotoid moves}. 
It is clear that if both forbidden moves were allowed, any knotoid diagram in $S^2$ and in $\mathbb{R}^2$ could be turned into the trivial knotoid diagram. 
\begin{note}\normalfont
There are two situations where forbidden moves seemingly occur. Precisely, when the arc adjacent to an endpoint is involved in an $\Omega_1$- or an $\Omega_2$-move followed by a planar isotopy, as illustrated in Figure \ref{fig:fake}.  These moves have the same effect as one or two consecutive forbidden moves of same type. We shall call these moves \textit{fake forbidden moves}.
\end{note}
\begin{figure}[H]
 \centering 
\includegraphics[width=.7\textwidth]{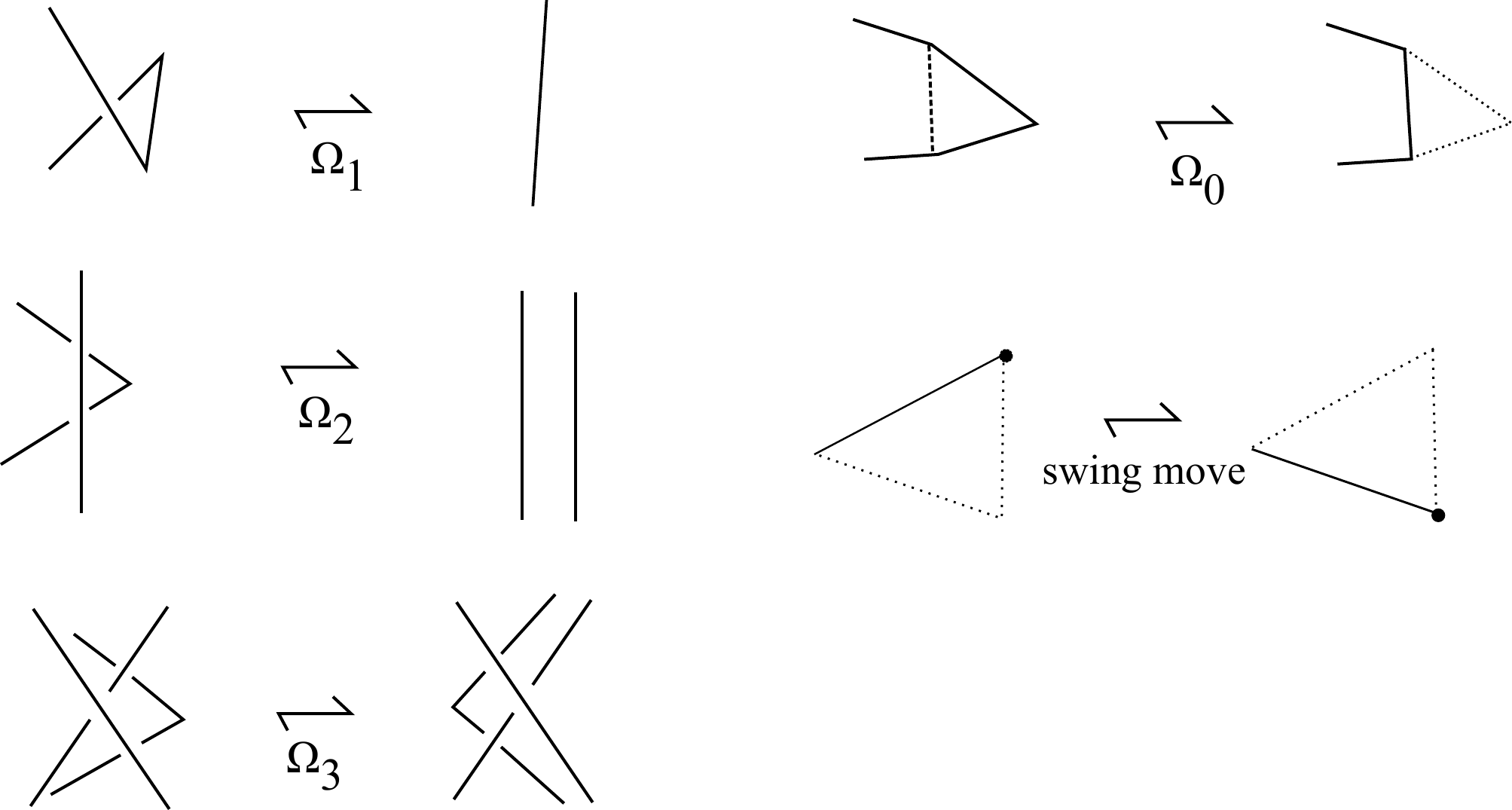}
\caption{$\Delta$-moves}
 \label{fig:om}
\end{figure}

	\begin{figure}[H]
		\Large{
        \centering  \scalebox{0.4}{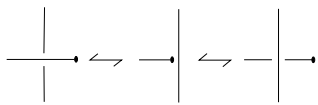}
				}
        \caption{Forbidden knotoid moves}
        \label{subfig:for}
		\end{figure}
		
		\begin{figure}[H]
		\Large{
        \centering  \includegraphics[width=1\textwidth]{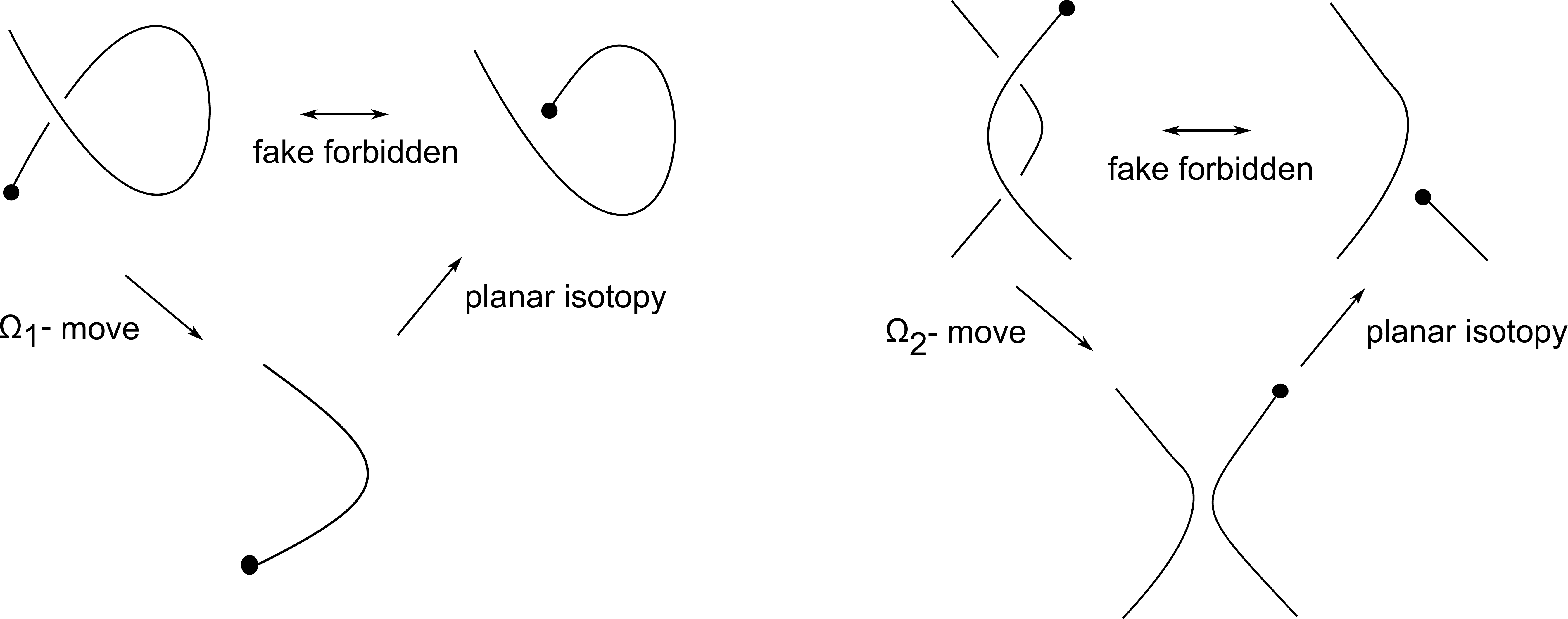}
				}
        \caption{Fake forbidden moves}
        \label{fig:fake}
		\end{figure}
		
 We shall call $\Omega_0, \Omega_1,\Omega_2,\Omega_3$ moves together with the swing moves the $\Omega$-{\it moves}. Two knotoid diagrams are \textit{isotopic} to each other if there is a finite sequence of $\Omega$-moves that transforms one into the other. The isotopy generated is clearly an equivalence relation and the isotopy classes of knotoid diagrams are called \textit{knotoids in} $\mathbb{R}^2$. The set of all knotoids in $\mathbb{R}^2$ is denoted by $\mathcal{K}(\mathbb{R}^2)$.

\subsection{Extending the definition of knotoids}
 A \textit{multi-knotoid diagram} in $S^2$ or in $\mathbb{R}^2$ \cite{Tu} is an extended knotoid diagram having a finite number of knot diagrams. The equivalence relation defined on knotoid diagrams applies to multi-knotoid diagrams directly, and the corresponding equivalence classes are called \textit{multi-knotoids}. 
Note that other notions generalizing knotoid can be introduced. For instance we can also define a \textit{linkoid diagram} as an immersion of a number of oriented intervals as a generalizing concept to $(n, n)$- tangles. See \cite{Gthesis} for a discussion on such generalizations.

In this paper we work with multi-knotoids in $\mathbb{R}^2$.


\section{Braidoids}{\label{sec:braidoids}}

In this section we define braidoid diagrams and the isotopy classes of them that we call braidoids. A braidoid diagram is defined similarly to a braid diagram as a system of finite descending strands. The main difference is that a braidoid diagram has one or two of its strands starting with/terminating at an endpoint that is not necessarily at top or bottom lines of the defining region of the diagram.

\subsection{The definition of a braidoid diagram}

\begin{definition}\normalfont \label{braidoiddefn}
Let $I$ denote the unit interval $[0,1] \subset \mathbb{R}$. 
A \textit{braidoid diagram} $B$ is a system of a finite number of arcs immersed in $I \times I\subset \mathbb{R}^2$. We identify $\mathbb{R}^2$ with the $xt$-plane with the $t$-axis directed downward. The arcs of $B$ are called the \textit{strands} of $B$.  Following the natural orientation of $I$, each strand is naturally oriented downward, with no local maxima or minima. There are only finitely many intersection points among the strands,  which are  transversal double points endowed with over/under data. Such intersection points are called \textit{crossings} of $B$. 

 A braidoid diagram has two types of strands, the classical strands and the free strands. A \textit{classical strand} is like a braid strand connecting a point on $I \times\{0\}$ to a point on $I \times\{1\}$. A \textit{free strand} either connects a point in $I\times\{0\}$ or $I \times \{1\}$ to a special point that is located anywhere in $I \times I$ or connects two special points located anywhere in $I \times I$. These special points are called the \textit{endpoints} of $B$. A braidoid diagram contains either one free strand (see for example Figure~\ref{fig:basicbrds}(a)) or two free strands (see for example Figure~\ref{fig:basicbrds} (b), (c), (d), (e)) and exactly two endpoints. The endpoints  are specifically named as the \textit{leg} and the \textit{head} and are emphasized by graphical nodes labeled by $l$ and $h$, respectively, in analogy with the endpoints of a knotoid diagram. Precisely, the head is the endpoint that is terminal for a free strand with respect to the orientation, while the leg is the starting endpoint for a free strand with respect to the orientation.  See some examples of braidoid diagrams in Figure \ref{fig:basicbrds}.
\begin{figure}[H]
\centering \includegraphics[width=.85\textwidth]{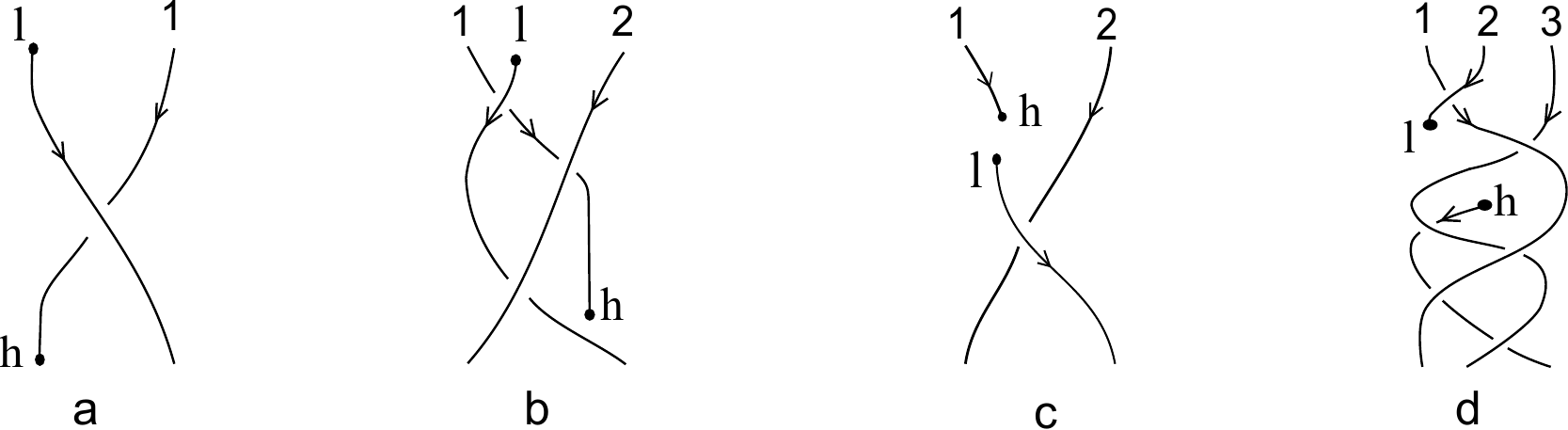} 
\caption{Some examples of braidoid diagrams}
\label{fig:basicbrds}
\end{figure}


The ends of the strands of $B$ other than the endpoints are called \textit{braidoid ends}. We assume that braidoid ends lie equidistantly on the top and the bottom lines and none of them is vertically aligned with any of the endpoints. It is clear that the number of braidoid ends that lie on the top line is equal to the number of braidoid ends that lie on the bottom line of the diagram. The braidoid ends on top and bottom lines are arranged in pairs so that they are vertically aligned and are called \textit{corresponding ends}. We number them with nonzero integers according to their horizontal order (from left to right), as in the examples illustrated in Figure~\ref{fig:basicbrds}. 
\end{definition}
Note that the endpoints of $B$ differ conceptually  from its braidoid ends. As we will see in the next section, the endpoints are subject to some isotopy moves unlike  the braidoid ends, and the endpoints do not participate in the closure operation we introduce unlike the braidoid ends. 
 
 A braidoid diagram is \textit{piecewise-linear} if all of its strands are formed by consecutive linear segments. Any braid diagram can be represented by a piecewise-linear braid diagram. Likewise, any braidoid diagram can be represented by  a piecewise-linear braidoid diagram.  We shall consider piecewise-linear braidoid diagrams, when convenient.

\subsection{Braidoid isotopy}\label{sec:isotopy}

There are two types of local moves generating the braidoid isotopy.

\subsubsection{Moves on segments of strands} 

We adapt the \textit{$\Delta$-moves} introduced in Section \ref{sec:knomo} to braidoid diagrams. A braidoid $\Delta$-\textit{move} replaces a segment of a strand with two segments in a triangular disk free of endpoints,  passing only over or under the arcs intersecting the triangular region of the move whilst the downward orientation of the strands is preserved (see Figure~\ref{fig:d}). The oriented $\Omega_0$, $\Omega_2$ and $\Omega_3$, which keep the arcs in the move patterns directed downward, can be viewed as special cases of braidoid $\Delta$-moves.  

\begin{figure}[H]
\centering  \scalebox{.2}{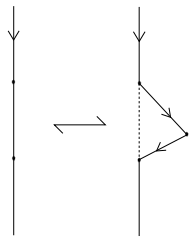}
\caption{A planar $\Delta$-move on a braidoid diagram}
\label{fig:d}
\end{figure}

\subsubsection{Moves of endpoints}
  
	Like for knotoid diagrams, we forbid to pull/push an endpoint of a braidoid diagram over or under a strand, as shown in Figure~\ref{fig:f}. These are {\it forbidden moves} on braidoid diagrams. It is clear that allowing both forbidden moves can cancel any braiding of the free strands. 

\begin{figure}[H]
\centering
\scalebox{.3}{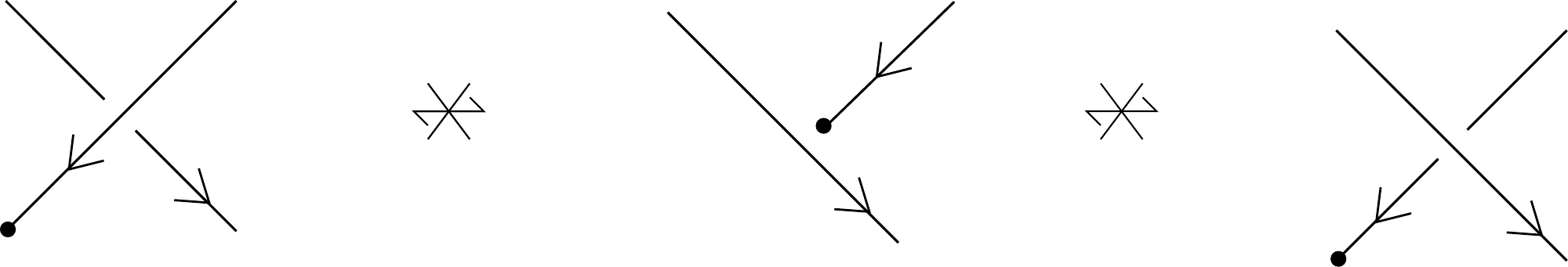}
\caption{Forbidden braidoid moves}
\label{fig:f}
\end{figure}

We allow the following moves on segments of braidoid strands containing endpoints.

	\begin{enumerate}
	\item \textit{Vertical Moves}:
	As shown in Figure \ref{fig:vertmo}, the endpoints of a braidoid diagram can be pulled up or down in the vertical direction as long as they do not violate any of the forbidden moves (e.g. crossing through or intersecting any strand of the diagram). Such moves are called \textit{vertical moves}. 
 \item \textit{Swing Moves}: An endpoint can also swing to the right or the left like a pendulum (see Figure \ref{fig:brswing2}) 
as long as the downward orientation on the moving arc is preserved, and the forbidden moves are not violated.
\end{enumerate}
\begin{figure}[H]
\centering
\includegraphics[width=.4\textwidth]{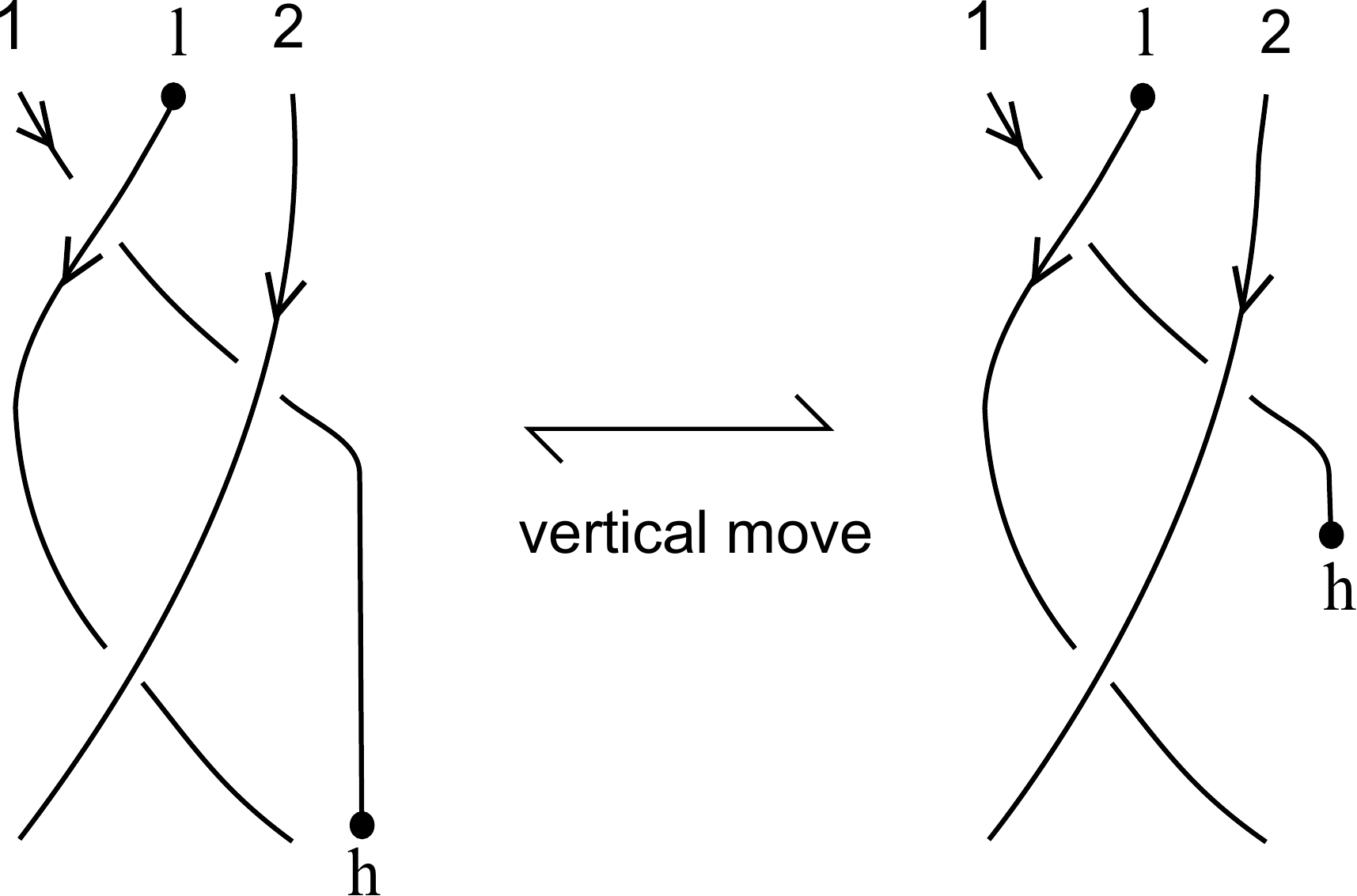}
\caption{A vertical move on $h$}
\label{fig:vertmo}
\end{figure}
\begin{figure}[H]
\centering
\includegraphics[width=.5\textwidth]{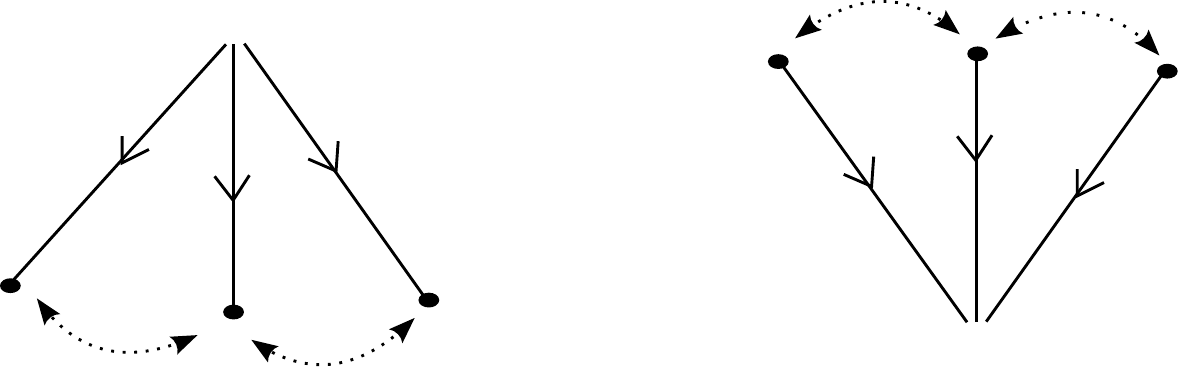}
\caption{The swing moves for braidoids}
\label{fig:brswing2}
\end{figure}

\begin{definition}[braidoid isotopy]\normalfont\label{defn:isotopy}
 It is clear that assuming braidoid ends fixed at the top and bottom lines, the braidoid $\Omega_2$ and $\Omega_3$- moves together with braidoid $\Omega_{0}$-moves and the swing and vertical moves for the endpoints generate an equivalence relation on braidoid diagrams in $\mathbb{R}^2$. Two braidoid diagrams are said to be \textit{isotopic} if one can be obtained from the other by a finite sequence of braidoid isotopy moves. An isotopy class of braidoid diagrams is called a {\it braidoid}. 
\end{definition}


\section{From braidoids to planar knotoids - a closure operation}\label{sec:closure}

 We present a closure operation on braidoid diagrams in analogy with the closure of braids in handlebodies \cite{LO}. In order to do this, we introduce the notion of labeled braidoids.

 \begin{definition}\normalfont
A \textit{labeled braidoid diagram} is a braidoid diagram whose corresponding ends are labeled either with $o$ or $u$ in pairs. See Figures \ref{fig:c}, \ref{fig:cll}. Each label indicates either an \textit{overpassing} or \textit{underpassing} arc, respectively, that will take place in the closure operation explained below.
\end{definition}

 \begin{definition}\normalfont \label{closure}
Let $B$ be a labeled braidoid diagram. The \textit{closure} of $B$, denoted $\widehat{B}$, is a planar (multi)-knotoid diagram obtained by the following  topological operation: each pair of corresponding ends of $B$ is connected with an embedded arc (with slightly tilted extremes) that runs along the right hand-side of the vertical line passing through the ends and in a distance arbitrarily close to this line. The connecting arc goes {\it entirely over} or {\it entirely under} the rest of the diagram according to the label of the ends. We demonstrate abstract and concrete examples in Figure~\ref{fig:c} and~\ref{fig:cll}, respectively. 
\end{definition}
\begin{figure}[H]
\centering 
\includegraphics[width=.15\textwidth]{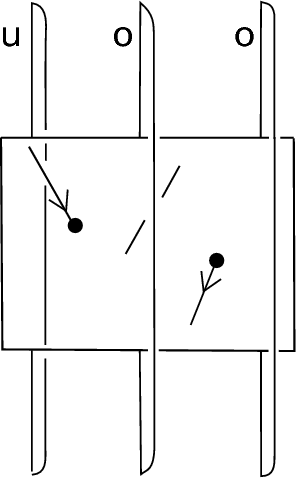}
\caption{The closure of an abstract labeled braidoid diagram}
\label{fig:c}
\end{figure}

The reason that a joining arc is required to lie in an arbitrarily close distance to the line of the related corresponding ends is that, otherwise,  forbidden moves may obstruct an isotopy of $\widehat{B}$ between any two joining arcs. Notice also that the resulting multi-knotoid depends on the labeling of the braidoid ends. In Figure~\ref{fig:cll} we see two labeled braidoid diagrams induced by the same underlying braidoid diagram via different labelings, give rise to non-equivalent knotoids. 
 
On labeled braidoid diagrams we allow the braidoid $\Delta$-moves and the vertical moves to take place on labeled braidoid diagrams and we forbid the forbidden braidoid moves. However, we do not allow swing moves for labeled braidoid diagrams in full generality. We only allow the {\it restricted swing moves} whereby the swinging of an endpoint takes place within the interior of the vertical strip determined by the neighboring vertical lines passing through two consecutive pairs of corresponding braidoid ends (see Figure \ref{fig:brswing}).
The reason for restricting the swing moves is because if the endpoints surpass the vertical lines of the corresponding ends this will cause forbidden moves on the closure.  
\begin{definition}\normalfont
 \textit{Labeled braidoid isotopy} is generated by the braidoid $\Omega$-moves, the vertical moves and the restricted swing moves, preserving at the same time the labeling.
Equivalence classes of labeled braidoid diagrams under this isotopy relation are called \textit{labeled braidoids}. 
\end{definition}  

\begin{figure}[H]
\centering
\includegraphics[width=.5\textwidth]{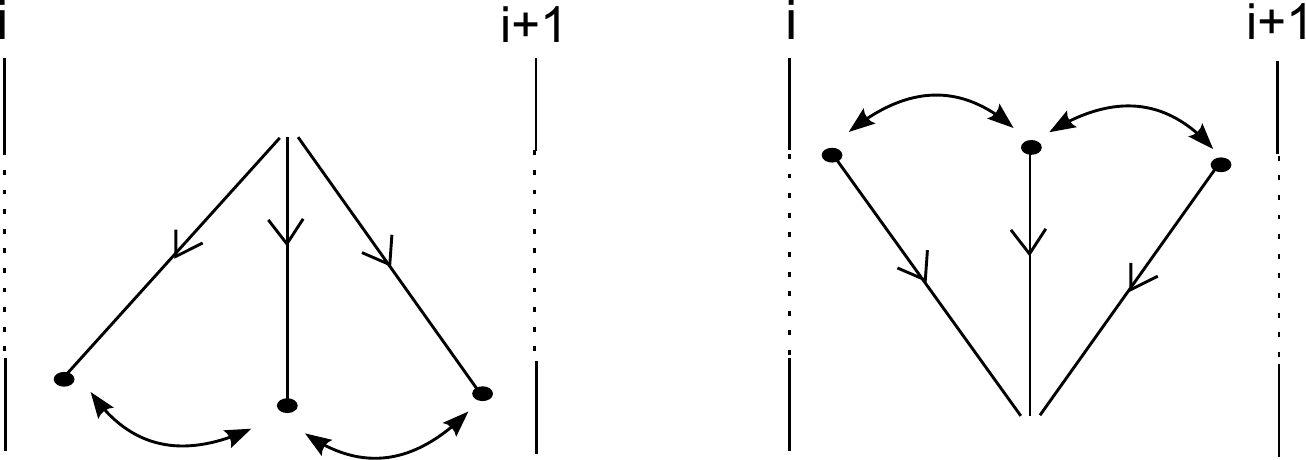}
\caption{The restricted swing moves for braidoids}
\label{fig:brswing}
\end{figure}

\begin{figure}[H]
\centering
\includegraphics[width=1\textwidth]{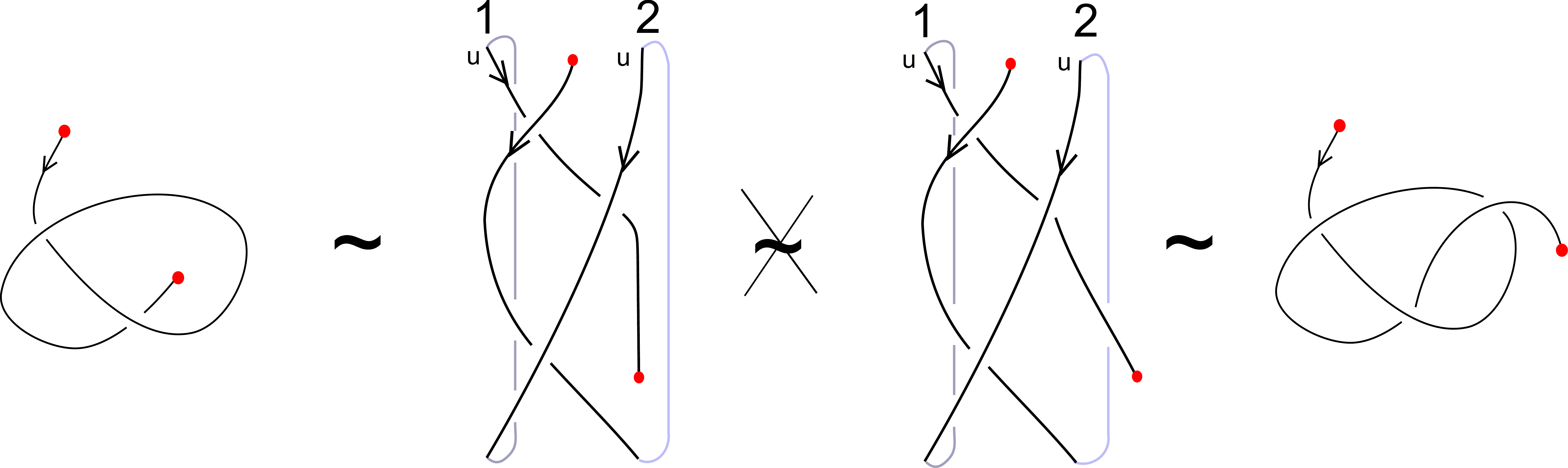}
\caption{The restricted swing moves for braidoids}
\label{fig:rest}
\end{figure}

\begin{lemma}
The closure operation induces a well-defined mapping from the set of labeled braidoids to the set of multi-knotoids in $\mathbb{R}^2$.
\end{lemma}
\begin{proof}
Let $b_1$ and $b_2$ be two labeled braidoid diagrams representing the same labeled braidoid. Then $b_1$ and $b_2$ differ from each other by braidoid isotopy moves. It is clear that braided $\Omega_2$ and $\Omega_3$-moves  are transformed into a sequence of knotoid $\Omega_2$ and $\Omega_3$-moves. Also, the vertical and swing moves are transformed into planar isotopy on the (multi)-knotoid diagram obtained by the closure. Therefore the closures of $b_1$ and $b_2$ are isotopic (multi)-knotoid diagrams. 

\end{proof}

 

\begin{figure}[H]
\centering 
\includegraphics[width=.9\textwidth]{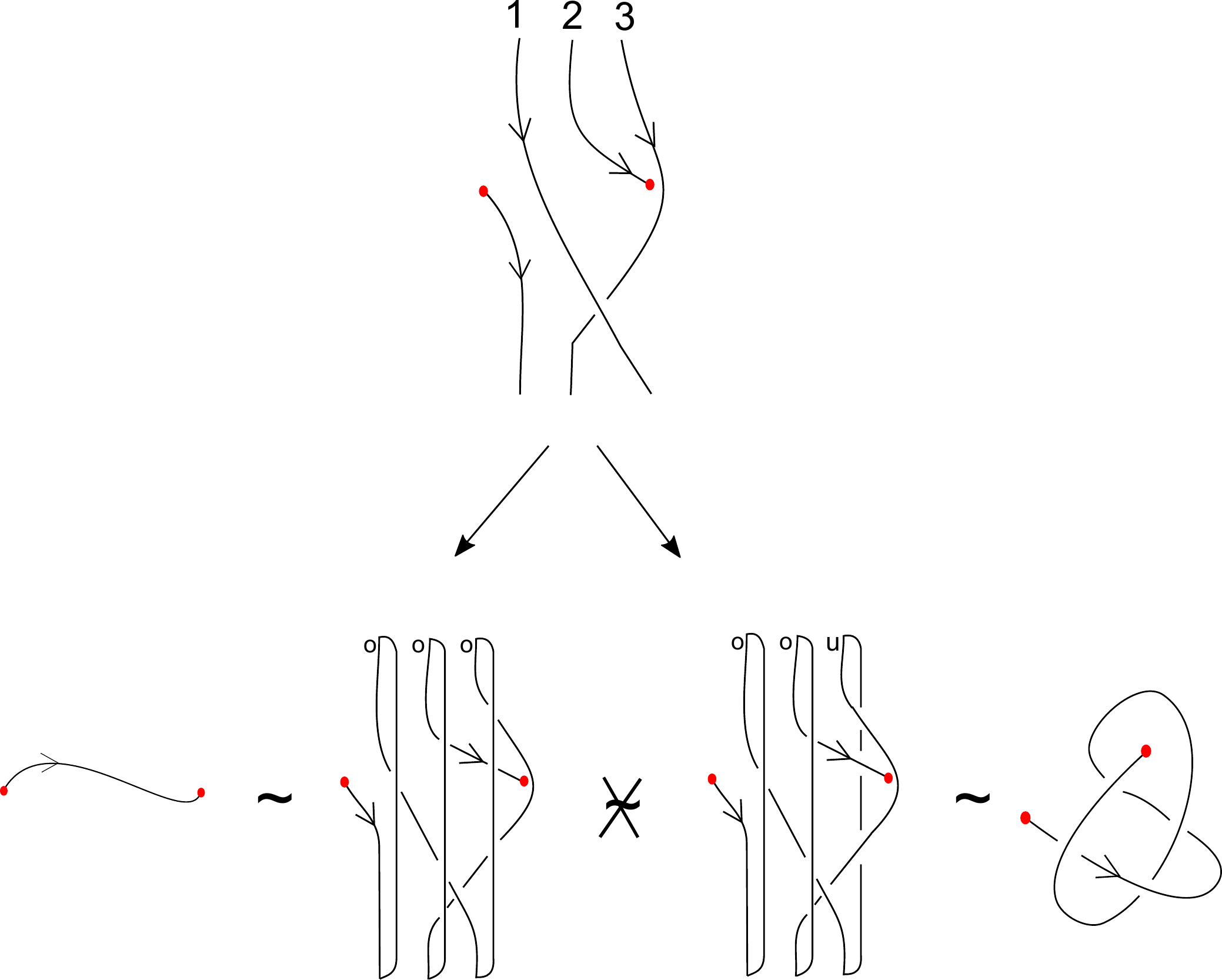}
\caption{An example of non-equivalent labeled closures}
\label{fig:cll}
\end{figure}
\section{An algorithm for obtaining labeled braidoids from planar knotoids}\label{sec:braiding}
In this section we present the braidoiding moves on knotoid diagrams that induce algorithms for turning any planar (multi)-knotoid diagram into a labeled braidoid diagram \cite{Gthesis}. By these algorithms we obtain the following theorem.

\begin{theorem}(An analogue of the Alexander theorem for knotoids) \label{thm:alex}
Any (multi)-knotoid diagram in $\mathbb{R}^2$ is isotopic to the closure of a labeled braidoid diagram.
\end{theorem}

In \cite{Gthesis} two such algorithms for proving Theorem \ref{thm:alex} are presented. One of the algorithms appearing also in \cite{GL1} is conceptually lighter than the other one presented here, see also Remark \ref{alg}. The algorithm we present here is more `rigid' in the sense that it assumes knotoid diagrams as rigid diagrams. This makes the algorithm more appropriate for proving a braidoid equivalence result analogous to the classical Markov theorem for classical braids.

\subsection{The basics of the braidoiding algorithm}
Let $K$ be a (multi)-knotoid diagram in a plane identified with the $xt$-plane. We describe below how to manipulate $K$ in order to obtain a labeled braidoid diagram, after endowing the plane of $K$ with top-to-bottom direction. We will be assuming that $K$ lies in $[0,1]\times[0,1]$ since $K$ is compact.

\subsubsection{Up-arcs and free up-arcs} 
It is clear that by small perturbations $K$ can be assumed to be a diagram  without any horizontal or vertical arcs. Thus $K$ consists of a finite number of arcs oriented either upward or downward, and these arcs are separated by finitely many local maxima or minima. The  arcs of $K$ that are oriented upward are called \textit{up-arcs} and the ones oriented downward are called \textit{down-arcs} of $K$. An up-arc may contain crossings of different types (over/under-crossings) or no crossing at all. See Figure \ref{fig:up}.  An up-arc that contains no crossing is called a \textit{free up-arc}. 
\begin{figure}[H]
\centering
\includegraphics[width=.75\columnwidth]{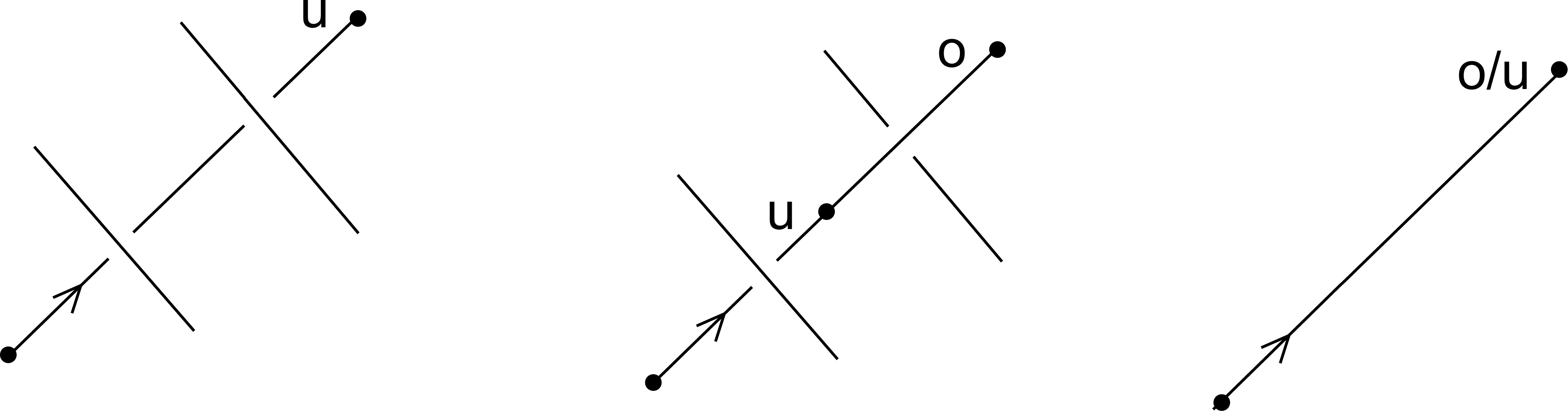}
\caption{ Two up-arcs containing crossings and a free up-arc}
\label{fig:up}
\end{figure}

\subsubsection{Subdivision} \label{sec:subdivi}
We start by marking the local maxima and minima of $K$ with points, which we name as {\it subdividing points}. 
In the process we may need to subdivide further some of the up-arcs of $K$ so that each one contains crossings of only one type. We attach a label to each up-arc accordingly to the crossing type it contains: we attach $o$ if the up-arc contains over-crossing(s), $u$ if the up-arc contains under-crossing(s). The up-arcs that are free of crossings can be labeled either $o$ or $u$.   

\subsubsection{Braidoiding moves}
The basic idea of turning $K$ into a labeled braidoid diagram is to keep the arcs of $K$ that are oriented downward, with respect to the top-to-bottom direction, and to eliminate its up-arcs by turning them into braidoid strands. 
The elimination of the up-arcs is done by utilizing the sequence of moves illustrated in Figure ~\ref{fig:brdngone}. Precisely, a \textit{braidoiding move} consists of cutting an up-arc at a point and pulling the resulting sub-arcs to top and bottom lines {\it entirely over} or {\it entirely under} the rest of the diagram by preserving the alignment with the cut-point. Finally we slide the resulting sub-arcs down and up, respectively, across local triangular regions in order to eliminate the upward oriented pieces. An up-arc is eventually turned into a braidoid strand as also depicted in Figure \ref{fig:brdngone}. It can also be verified by Figure \ref{fig:brdngone} that the resulting ends obtained by cutting the up-arc $QP$ at a point are pulled entirely over the rest of the diagram and received the label $o$, and when we join the resulting pair of corresponding ends with an over-passing arc, we obtain a closed strand that is isotopic to the initial up-arc $QP$. 
\begin{figure}[H]
\centering
\scalebox{.55}{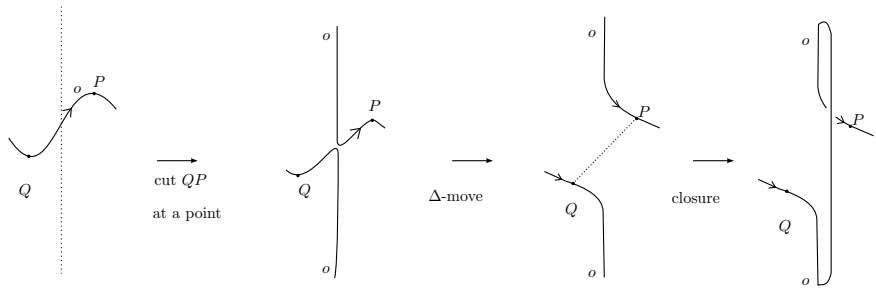}
\caption{A braidoiding move}
\label{fig:brdngone}
\end{figure}
\subsubsection{Cut-points and sliding triangles} 
Let $QP$ be an up-arc of $K$ with respect to a given subdivision of $K$ where $Q$ denotes the initial and $P$ denotes the top-most subdivision point. 

The right angled triangle which lies below $QP$ and admits $QP$ as hypotenuse, and is a special case of a triangle enabling a $\Delta$-move is called the \textit{sliding triangle} of the up-arc $QP$. We denote the sliding triangle by $T(P)$, see Figure \ref{fig:slid}.
The disk bounded by the sliding triangle $T(P)$ is utilized after cutting $QP$, for sliding down the resulting lower sub-arc across it. See Figure \ref{fig:slid}.   	
\begin{figure}[H]
\centering
\includegraphics[width=.75\columnwidth]{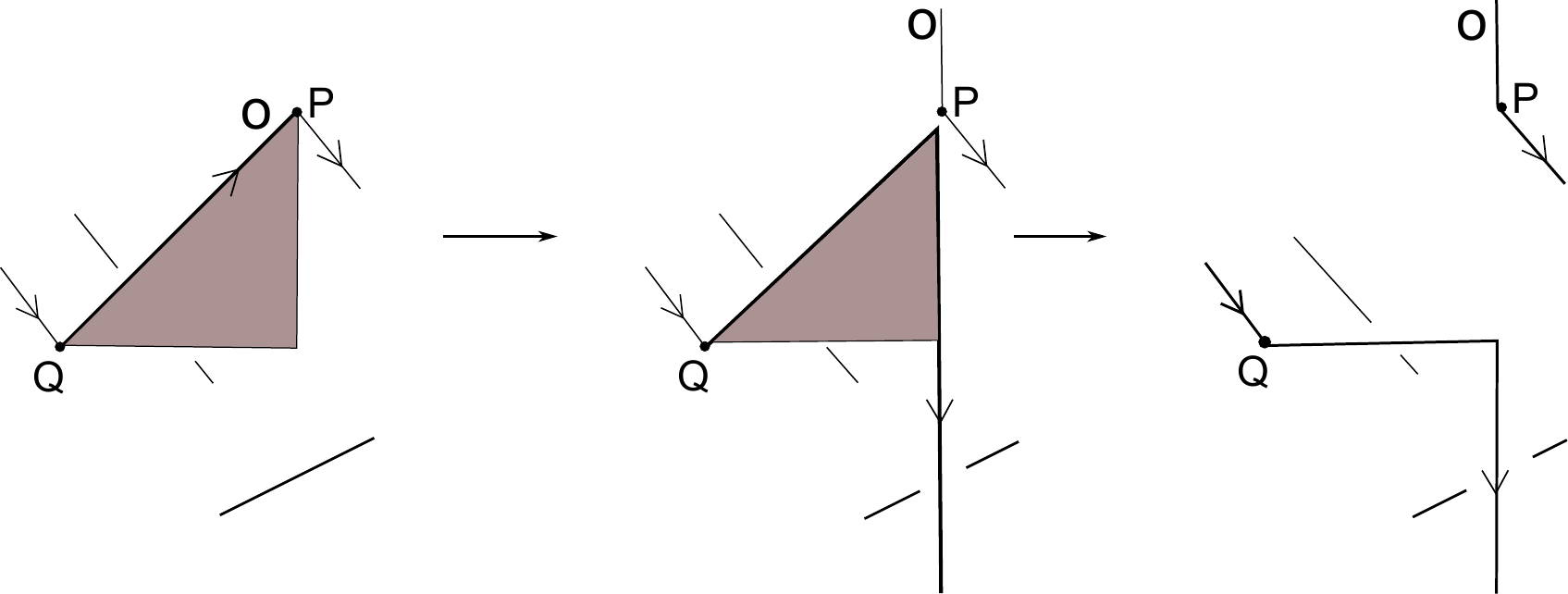}
\caption{The sliding triangle of the up-arc $QP$ with $P$ the cut point}
\label{fig:slid}
\end{figure}
A \textit{cut-point} of an up-arc is defined to be the point where the up-arc is cut to start a braidoiding move. We pick the top-most point $P \in QP$ as the cut-point of $QP$ for our algorithm. 

\subsubsection{A condition on sliding triangles} 
One may come across the  situation where one (or two) of the endpoints lies in the region bounded by some sliding triangle as in Figure \ref{fig:endptcond}. This is an unwanted situation since it would cause a forbidden move when the knotoid strand obtained by closing the resulting braidoid strand is tried to be isotoped back to the initial up-arc. We impose the following condition on knotoid diagrams to avoid this.\\

 \textit{The endpoint triangle condition} A sliding triangle of a knotoid diagram is not allowed to contain an endpoint.\\

To satisfy this condition we introduce a subdivision of up-arcs into smaller sub-arcs by adding extra subdividing points. See Figure \ref{fig:endptcond}. More precisely, we have the following proposition. 
\begin{lemma}
Let $K$ be a knotoid diagram and $T(P)$ 
be the sliding triangle corresponding to an up-arc $QP$.
If $T(P)$ contains the leg or the head of $K$ in its interior or boundary then there is a further subdivision of $QP$ admitting new sliding triangles whose disks are disjoint from the endpoint.
\end{lemma}
\begin{proof}
 We can assume that $QP$ lies in the first quadrant of the plane and has a positive slope without loss of generality.
 There is a unique horizontal and vertical line passing through the endpoint in question and each intersecting $QP$ exactly at one point, since $x_Q \leq x_{endpoint} \leq x_P$ and $t_P \leq t_{endpoint} \leq t_Q$. We pick an interior point in the small line segment whose boundary is the union of the two intersection points and declare it as a new subdividing point on $QP$. Let $P^*$ denote the chosen point. The sliding triangles intersecting at the point $P^*$ are clearly smaller than the one with the top vertex $P$ and they do not contain the endpoint.
In case that  $T(P)$ contains two of the endpoints, we introduce two new subdividing points on $QP$ each chosen as above and none of the corresponding sliding triangles contains the endpoints.
\end{proof}  
\begin{figure}[H]
\centering
\includegraphics[width=.6\textwidth]{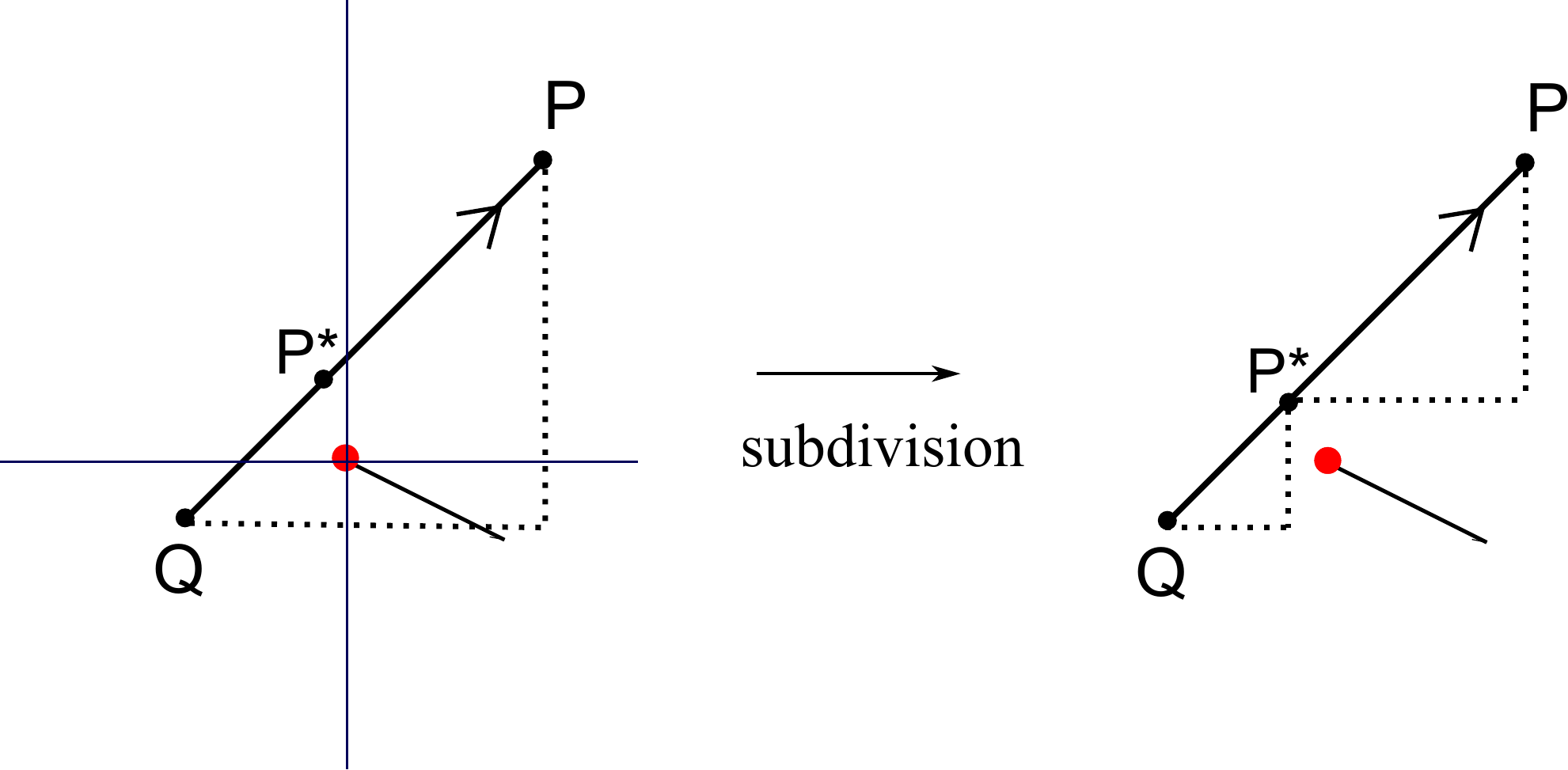}
     \caption{A further subdivision of the up-arc  $QP$}
      \label{fig:endptcond}
\end{figure}
\subsection{Braidoiding algorithm}
We now present our algorithm for obtaining a labeled braidoid diagram from a (multi)-knotoid diagram $K$. The algorithm runs as follows.

{\it Step 1: Preparation for eliminating the up-arcs}
\smallbreak
\begin{enumerate}
\item The diagram is marked with subdividing points as explained in Section \ref{sec:subdivi} so that the endpoint triangle condition is satisfied. It is furthermore assumed to satisfy some general positioning conditions, which can be ensured by `small' isotopy moves, namely
\item no arcs is vertical or horizontal,
\item no subdividing points are vertically aligned with each other unless they share a common edge and neither with the endpoints or with any of the crossings,
\item the endpoints of $K$ are assumed to lie on arcs that are directed downward.
\end{enumerate}
\vspace{3pt}
{\it Step 2: Applying the braidoiding moves}
\smallbreak
 We order the up-arc and finally apply the braidoiding moves to each up-arc of $K$ in the given order. 
\smallbreak
Figure \ref{fig:algo} illustrates for the braidoiding algorithm with a concrete example.
\begin{figure}[H]
\centering
\includegraphics[width=1\textwidth]{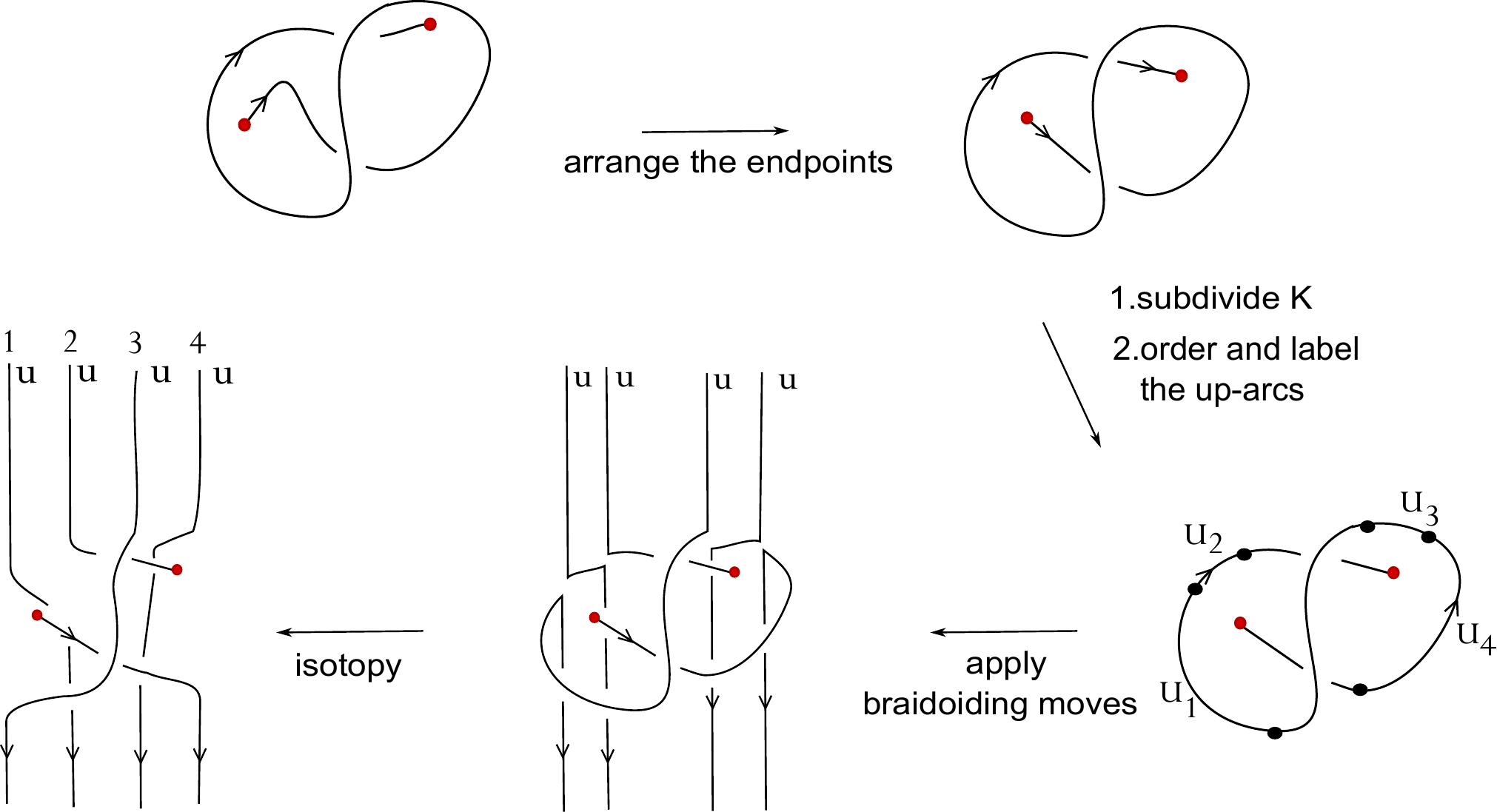}
\caption{An illustration for the algorithm}
\label{fig:algo}
\end{figure}

\subsubsection{Obstructions for the braidoiding algorithm and resolutions}
Now, we discuss some bugs of the braidoiding algorithm. In some cases, as exemplified in Figure \ref{fig:casetwo}, the algorithm is obstructed by a clasp occuring in the sliding triangle of an up-arc. We see in the figure that the braidoiding move applied on the second-ordered up-arcs labeled with $o$ and $u$ respectively, cannot be completed, due to the clasps in their sliding triangles. It can be verified by checking all possible positionings, labelings and orderings of any two up-arcs that this type of obstruction may occur only if:
\begin{itemize}
\item the top-most point of the first-ordered up-arc intersects the sliding triangle of the second-ordered up-arc and,\\
\item two up-arcs with intersecting sliding triangles, are labeled the same. 
\end{itemize}

\begin{figure}[H]
\centering
\includegraphics[width=.7\columnwidth]{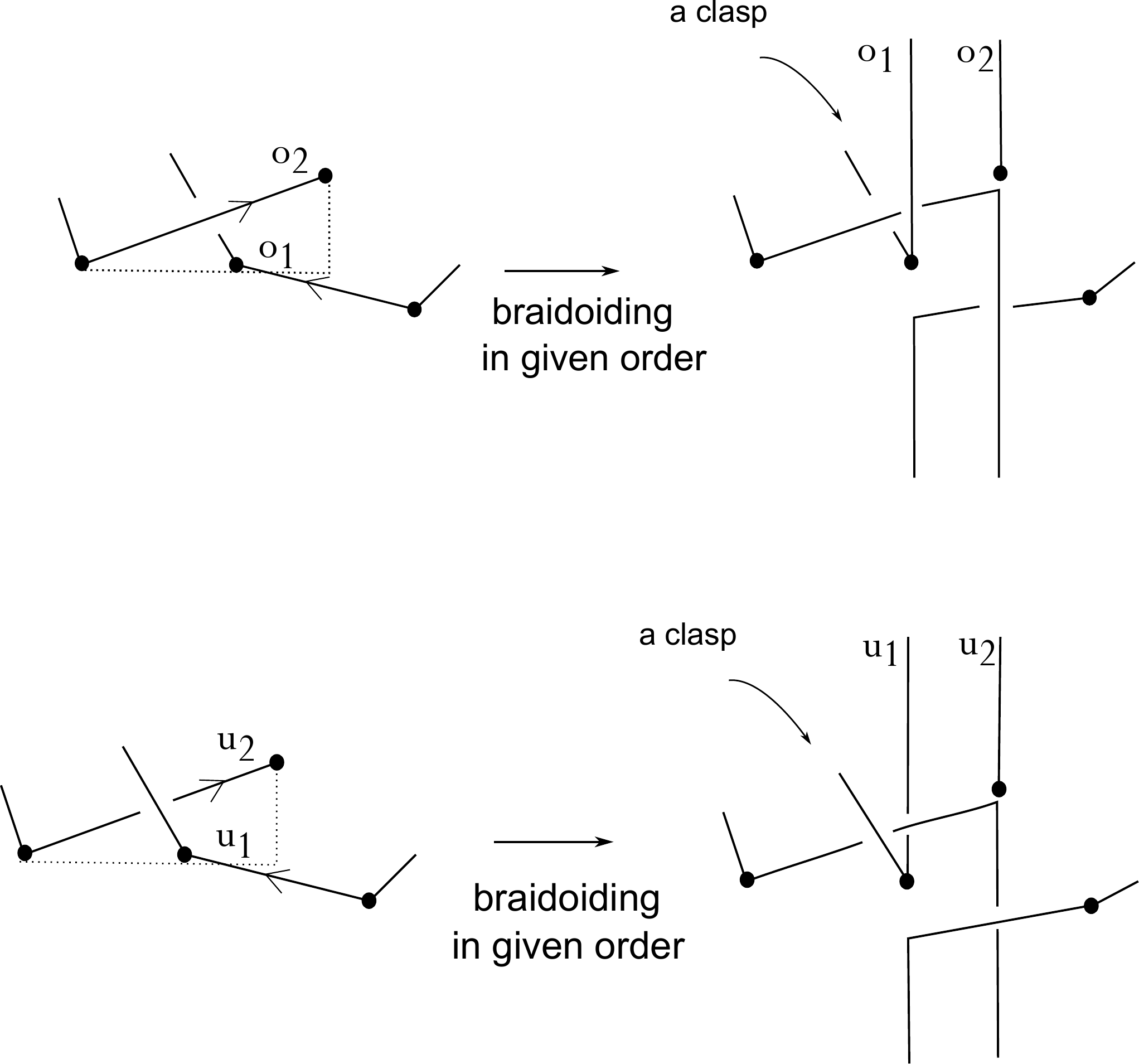}
\caption{Obstructions for applying  braidoiding moves}
\label{fig:casetwo}
\end{figure}
\subsubsection{Resolutions of the obstructions}
Due to the conditions creating obstructions, the resolutions for them can be:
\begin{itemize}
\smallbreak 
\item swapping the ordering of the up-arcs; see Figure \ref{fig:repair},\\
\item changing the label of the free up-arc if there is a free up-arc involved in an obstruction; see Figure \ref{fig:repairtwo},\\
\item subdividing the up-arcs further to have disjoint sliding triangles; see Figure \ref{fig:repairthree}.
\end{itemize}
\begin{figure}[H]
\centering 
\includegraphics[width=.6\columnwidth]{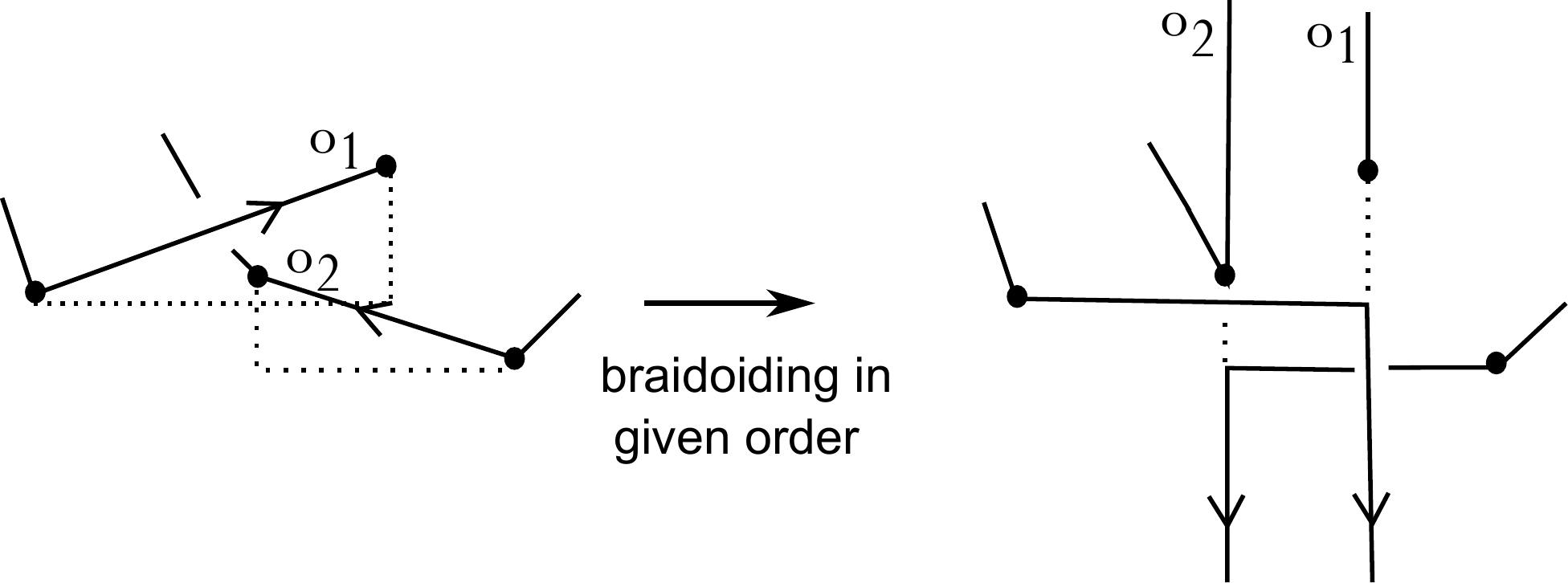}
\caption{Swapping the order of up-arcs repairs the obstruction}
\label{fig:repair}
\end{figure}
\begin{figure}[H]
\centering 
\includegraphics[width=.9\columnwidth]{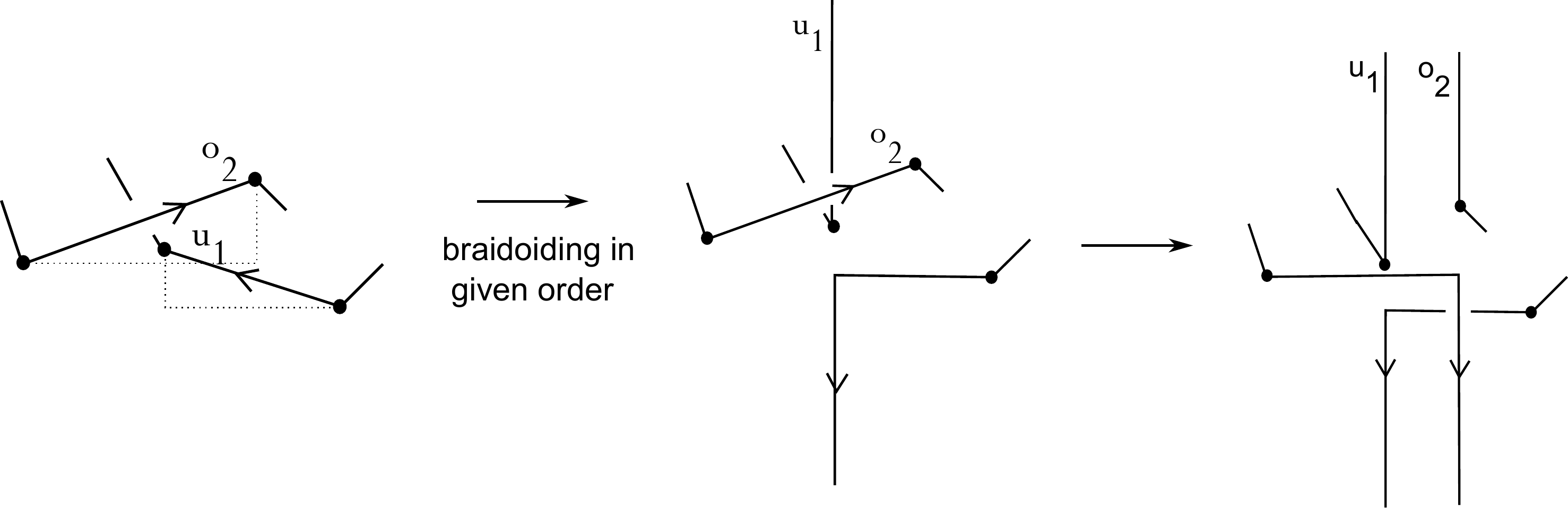}
\caption{Changing the label of the free up-arc repairs the obstruction}
\label{fig:repairtwo}
\end{figure}
\begin{figure}[H]
\centering 
\includegraphics[width=.75\columnwidth]{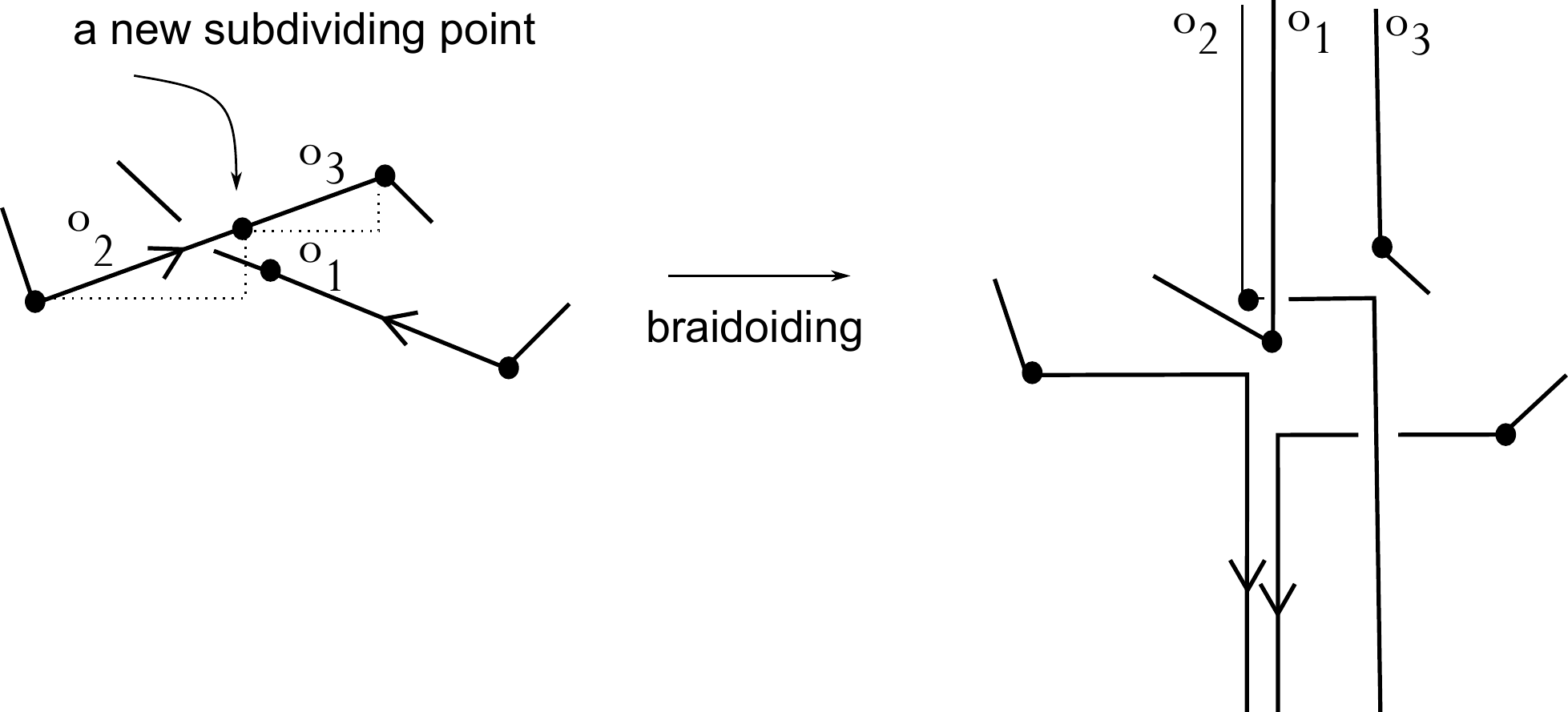}
\caption{Adding a subdividing point to the upper up-arc repairs the obstruction}
\label{fig:repairthree}
\end{figure}

\subsubsection{The classical triangle condition}
We can make the braidoiding algorithm a simultaneous algorithm that is processed independently of the ordering of the up-arcs, by imposing the following condition.
\begin{definition}\normalfont
Two sliding triangles are said to be \textit{adjacent} if the corresponding up-arcs have a common subdidiving point, and \textit{non-adjacent} otherwise.
\end{definition}
\smallbreak
The {\it classical triangle condition} says that non-adjacent sliding triangles are allowed to intersect only if the up-arcs of the triangles have different labels.
 The classical triangle condition can always be satisfied by the following lemma.
 
\begin{figure}[H]
\centering\includegraphics[width=.5\columnwidth]{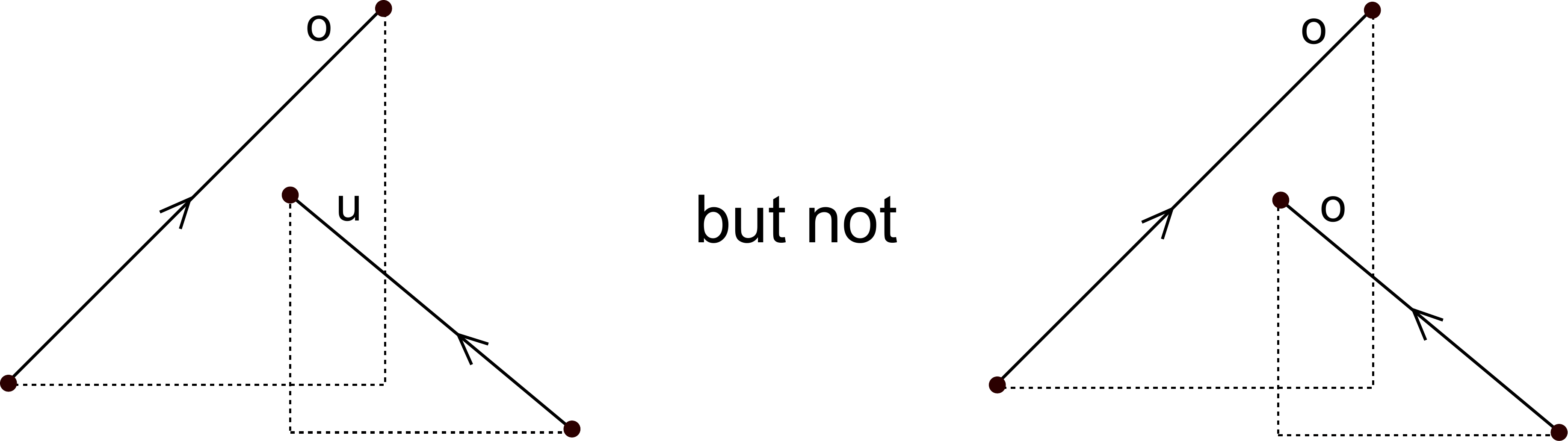}
\caption{The classical triangle condition}
\label{fig:clascond}
\end{figure}

\begin{lemma}\label{lem:triii}
Let $K$ be a knotoid diagram. There exists a subdivision of $K$ satisfying both the classical triangle condition and the endpoint triangle condition.
\end{lemma}

\begin{proof}
This lemma is proved similarly with Lemma 1 in \cite{Lathesis}. For completeness, we adapt the proof here for the knotoid case.
Let $d_1$ be the minimum distance between any two crossings appearing on the up-arcs of $K$. Choose some $r$, $0 < r < d_1$ so that the disk of radius $r$ centered at a crossing contained in an up-arc, intersects the up-arc at only four local strands around the crossing. Let $d_2$ be the minimum distance between any two disjoint points that are located outside the disks of radius $r$ around the crossings and on different segments. Letting $0 < \epsilon < \frac{1}{2}\hspace{0.1pt}min \{d_1, d_2\}$ be the distance between any two subdivision points on $K$ provides a subdivision of $K$ satisfying the classical triangle condition and also the triangle condition for the endpoints. Each sub-arc of length $\epsilon$ is also labeled according to the crossing type it contains, or if some become free, they are labeled freely. 
\end{proof}

\subsubsection{Proof of Theorem \ref{thm:alex}}

 With the discussion above we know that there is always a choice for ordering and labeling of up-arcs of a (multi)-knotoid diagram and, moreover that  we can impose the classical triangle condition to make the algorithm free of ordering. This is equivalent to saying that the braidoiding moves can be done all simultaneously and it is sufficient for ensuring that the algorithm terminates in a finite number of steps and always results in a labeled braidoid diagram. Finally, it is clear from Figure \ref{fig:brdngone} that if we apply closure to the corresponding strands according to their label, then each pair can be contracted back to the initial up-arc. Such contractions utilize the isotopy moves for knotoids, thus the closure of the resulting braidoid diagram is isotopic to the (multi)-knotoid diagram we started with. 
\hfill{QED}

\begin{remark}\label{alg}\normalfont
In \cite{GL1} we present another braidoiding algorithm which is also induced by the same braidoiding moves. The main difference of that algorithm from the one we describe in this paper is that it firstly eliminates each up-arc of a knotoid diagram that contains a crossing by a 90-degree or 180-degree rotation of the crossing. The braidoiding moves are applied to the resulting up-arcs that are all free up-arcs. Furthermore, since the labels of the free up-arcs are not forced they can be labeled only {\it u}. Applying the braidoiding algorithm on such a knotoid diagram generates a labeled braidoid diagram whose strands are all labeled $u$. This induces the idea of closing braidoid diagrams in a uniform way, that is, by using only underpassing arcs that run in arbitrarily close distance to the vertical lines of pairs of corresponding ends, for connecting them. We call this closure the \textit{uniform closure}. 
\end{remark}
  
	The above remark leads to the following result.
	
\begin{theorem}\cite{Gthesis,GL1}
Any multi-knotoid diagram in $\mathbb{R}^2$ is isotopic to the uniform closure of a braidoid diagram.

\end{theorem}

\section{$L$-equivalence of braidoid diagrams}\label{sec:L-eqv.}

The $L$-moves were originally defined for classical braid diagrams \cite{La1, Lathesis,LR1} by the second listed author, and they were used for proving a one-move analogue of the classical two-move Markov theorem~\cite{Ma, Wei, Bir, Ben, Mo, Tr, BM}. In the sequel, we adapt the $L$-moves for braidoid diagrams and use them for formulating a geometric analogue of the classical Markov theorem.
\subsection{The $L$-moves}
\begin{definition} \normalfont 
An {\it $L$-move} on a labeled braidoid diagram $B$ is the following operation:
\begin{enumerate}
\item Cut a strand of $B$ at an interior point which is not vertically aligned with a braidoid end, an endpoint or a crossing of $B$. The existence of such point can be ensured by a general positioning argument.

\item Pull the resulting ends away from the cut-point to top and bottom respectively, keeping them vertically aligned with the cut-point, so as to create a new pair of braidoid strands with corresponding ends. The new strands run both entirely {it over} or {\it under} the rest of the braidoid diagram depending on the type of the $L$-move applied. 
There are two types of $L$-moves, namely $L_{over}$ and $L_{under}$-moves, denoted by $L_o$ and $L_u$ respectively. An $L_o$-move comprises pulling the resulting sub-strands entirely over the rest of the diagram. An $L_u$-move comprises pulling the sub-strands entirely under the rest of the diagram. See  top row of Figure~\ref{fig:alfa}.
\item  After an $L$-move applied on a labeled braidoid diagram, the new pair of corresponding strands gets the labeling of the $L$-move: If the strands are obtained by an $L_o$-move then they are labeled $o$ and if they are obtained by an $L_u$-move then they are labeled $u$. Then, as can be verified by Figure~\ref{fig:alfa}, the closure of a pair of labeled braidoid strands resulting from an $L$-move is isotopic to the initial arc.
\end{enumerate}
\end{definition}



\begin{figure}[H]
\centering\scalebox{.8}{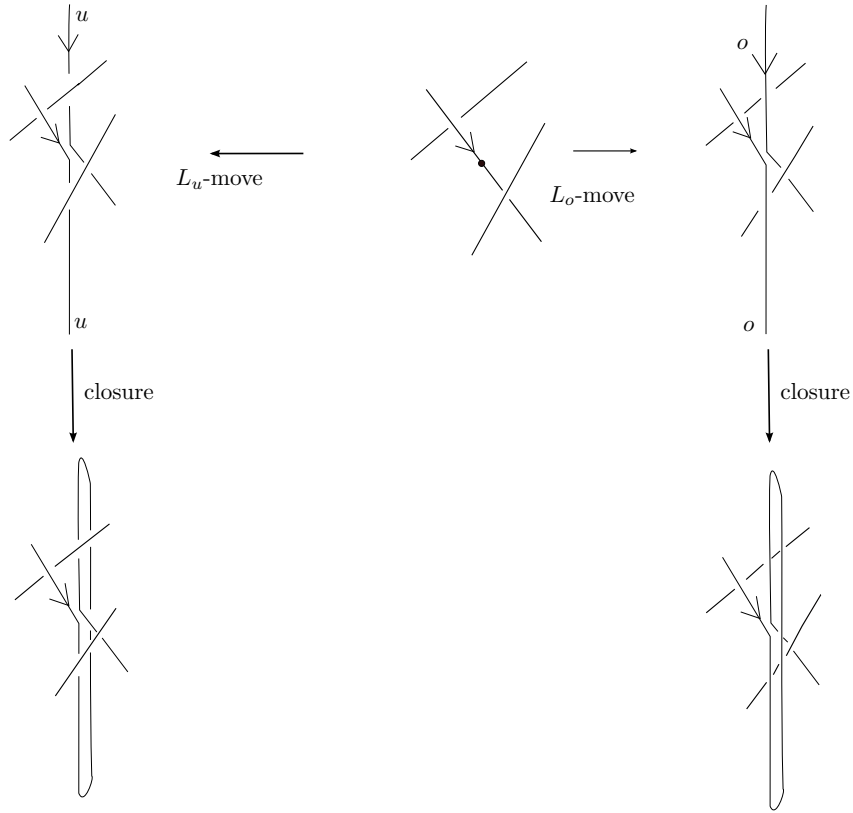}
\caption{$L$-moves and the closures of the resulting strands}
\label{fig:alfa}
\end{figure}

In order to formulate a braidoid equivalence in analogy to the classical braid equivalence utilized in Markov theorem we also need to discuss the fake forbidden moves.
\begin{definition}\label{defn:fake}\normalfont
We define a {\it fake forbidden move} on a labeled braidoid diagram $B$ to be a forbidden move on $B$ which upon closure induces a sequence of fake forbidden moves on the resulting (multi-)knotoid diagram. A {\it fake swing move} is a swing move which is not restricted, in the sense that the endpoint surpasses the vertical line of a pair of corresponding ends but in the closure it gives rise to a sequence of swing and fake forbidden moves on the resulting (multi-)knotoid diagram. See Figure \ref{fig:fakee} for an example of a fake swing move and a fake forbidden move on a labeled braidoid diagram.
\begin{figure}[H]
\centering\includegraphics[width=.85\columnwidth]{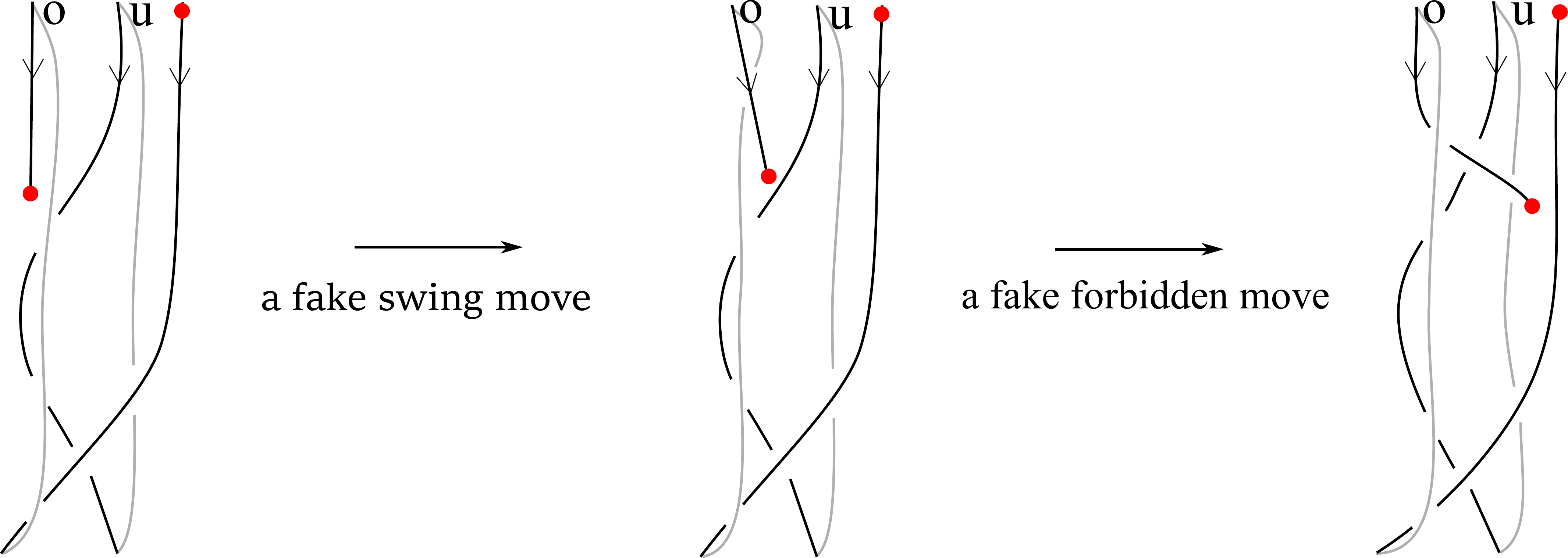}
\caption{A fake swing and a fake forbidden move}
\label{fig:fakee}
\end{figure}
\end{definition}
\begin{definition}\normalfont
The $L$-moves together with labeled braidoid isotopy moves and fake swing moves generate an equivalence relation on labeled braidoid diagrams that is called the {\it $L$-equivalence}.
 The $L$-equivalence is denoted by  $\sim_{L}$.
\end{definition}

 The labeled braidoid that is associated to a knotoid $K$ via a braidoiding algorithm, is not unique up to the braidoid isotopy. Precisely the labeled braidoid depends on the choices made for bringing the knotoid diagram to satisfy the general position requirements before starting the braidoiding algorithm, such as: arrangement of the endpoints, subdivision chosen on the arcs of the knotoid diagram and the labeling of the free up-arcs. Moreover, the knotoid diagram is subject to knotoid isotopy moves. We shall show that the labeled braidoid diagrams that are associated to $K$, are independent of algorithmic choices and isotopy moves up to $L$-equivalence.

Before stating our theorem we give the following definition.
\begin{definition}\normalfont
 A (multi)-knotoid diagram $K$ in the $xt$-plane, with a subdivision and labeling on its up-arcs, is said to be in \textit{general position} if 
\begin{itemize}
\item it has no vertical or horizontal arcs,
\item no subdividing point or endpoint is vertically aligned with a subdividing point, a crossing or an endpoint,
\item the arcs adjacent to the endpoints of $K$ are down-arcs,
\item no sliding triangle encloses an endpoint,
\item sliding triangles satisfy the classical triangle condition.
\end{itemize}
\end{definition}
 Note that a (multi)-knotoid diagram can be always brought to general position by small $\Delta$-moves. 

\begin{theorem} \cite{Gthesis}{(An analogue of the Markov theorem for braidoids)}{\label{thm:markov}}
The closures of two labeled braidoid diagrams are isotopic (multi)-knotoids in $\mathbb{R}^2$ if and only if the labeled braidoid diagrams are related to each other via  $L$-equivalence moves.
\end{theorem}
 \begin{proof}
For the proof of Theorem \ref{thm:markov}, we assume that (multi)-knotoid diagrams are in general position.

\smallbreak
The `if' part is clear from the definitions of $L$-moves and fake swing moves. More precisely, let $B_1, B_2$ be two labeled braidoid diagrams related to each other by an $L$-move. Let the arc, illustrated in Figure \ref{fig:alfa}, be a segment of a strand of $B_1$ on which an $L$-move is applied. It can be observed by the same figure that the closure of the resulting strands labeled accordingly to the type of the $L$-move applied, is isotopic to this arc. 
This implies that the closures of $B_1$ and $B_2$ are isotopic. From this, the closure map extends to a well-defined map $cl_L$ on the set of all $L$-equivalence classes of labeled braidoids.
\smallbreak
$cl_L:$ \{$L$-classes of labeled braidoid diagrams\}$\rightarrow$ \{Multi-knotoids in $\mathbb{R}^2$\}.\\
\smallbreak
For showing the `only if' part, we need to show that the map $cl_L$ is a bijection.  We adapt these parts here and check the cases involving the endpoints. For this we first show
the braidoiding algorithm induces a well-defined mapping $br$,\\
\begin{center}
$br$: \{Multi-knotoids in $\mathbb{R}^2$\} $\rightarrow$ \{$L$-classes of labeled braidoid diagrams\}, 
\end{center}
 that associates a (multi)-knotoid $K$ in $\mathbb{R}^2$ to the $L$-class of the braidoid diagram obtained from any representative of $K$ by the braidoiding algorithm. We call this map the \textit{braidoiding map}.
 
In order to show the mapping $br$ is well-defined we need to show that, up to $L$-equivalence, the resulting labeled braidoid diagram is independent of: the choices made for bringing a knotoid diagram into general position, choices of subdividing points for the up-arcs, labelings of free up-arcs, the $\Omega$-moves of knotoid diagrams. The parts of the proof of Theorem \ref{thm:markov} not involving the endpoints are analogous to the case of classical knots and braids \cite{La1, Lathesis}.
\begin{lemma}\label{lem:subd}
Let $K$ be a (multi)-knotoid diagram in $\mathbb{R}^2$ with a subdivision that satisfies the triangle conditions. Adding a subdividing point to an up-arc of $K$  and labeling the new up-arcs same as the initial labeling yields an $L$-equivalent labeled braidoid diagram.
\end{lemma} 
\begin{proof}
We first note that addition of a subdividing point does not violate the triangle conditions. Let $QP$ denote an up-arc of $K$ as depicted in Figure \ref{fig:pfprop}. First let us assume that the vertical line passing through an endpoint  of $K$ does not intersect $QP$. In this case, we can add a new subdividing point on $QP$ and apply braidoiding moves at that point and the point $P$. The proof for showing that the resulting braidoid diagram is $L$-equivalent to the one obtained by applying a braidoiding move at the top-most point $P$ follows similarly as for classical braid diagrams. The reader is directed to \cite{Lathesis} for details.

Now let us assume that the vertical line passing through an endpoint of $K$ intersects with $QP$ at some point on $QP$. Then $QP$ can be seen as the union of two sub-arcs joined at the intersection point and one of which contains the point $P$. The sub-arcs can be re-labeled accordingly to the initial labeling. By the classical argument discussed above, if a new subdividing point is added on $QP$ on the sub-arc containing $P$, then applying a braidoiding move at this point results in a labeled braidoid diagram that is $L$-equivalent to the labeled braidoid diagram obtained by applying a braidoiding move at the point $P$. Suppose now that a new subdividing point $P_1$ is chosen on the sub-arc that does not contain the point $P$.  
As we can verify in Figure \ref{fig:pfprop}, the labeled braidoid diagram resulting from a braidoiding move applied at the point $P$ can be taken to the braidoid diagram resulting from the latter subdivision (that is by applying braidoiding moves at the points $P$ and $P_1$) by an $L$-move that is applied at the point $Q^*$. Here the point $Q^*$ denotes the intersection point of the vertical line passing through $P_1$ and the lower strand containing $Q$ that has been obtained by the initial braidoiding of $QP$. Note that a neighborhood of $Q$ containing the point $Q^*$ is perturbed to slope slightly downwards for enabling application of the $L$-move.
\end{proof}
\begin{figure}[H]
\centering  \includegraphics[width=.7\columnwidth]{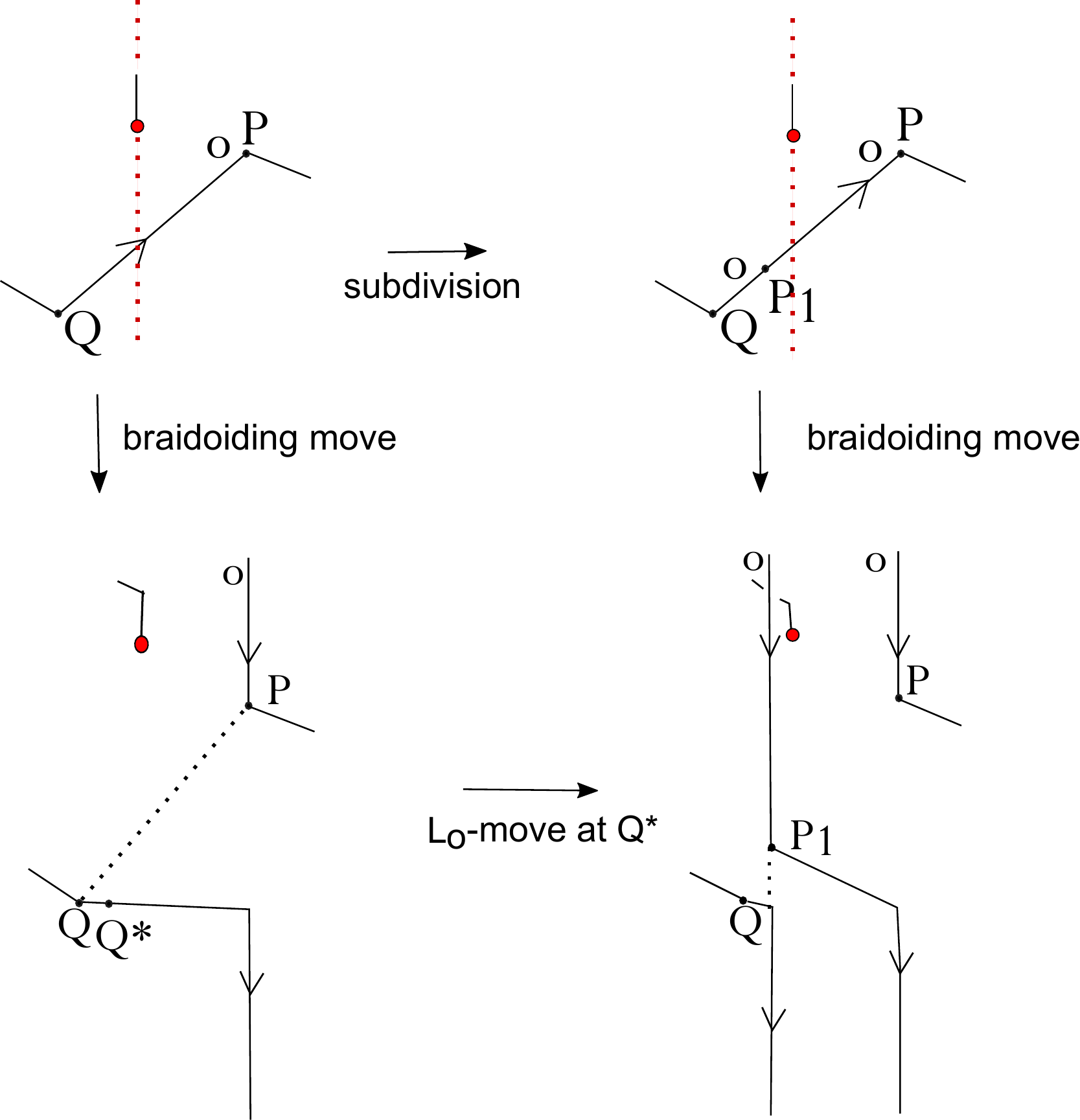}
\caption{Adding a subdividing point on the up-arc yields $L$-equivalence}
\label{fig:pfprop}
\end{figure}

\begin{lemma}\label{lem:free}
Labeling a free up-arc either with $o$ or $u$ does not change the resulting labeled braidoid diagram up to the $L$-equivalence.
\end{lemma}
\begin{proof}
The proof is illustrated in Figure \ref{fig:free} and it goes in analogy with the classical case. In the figure, a free up-arc $QP$ is labeled  both with $u$ and $o$. For simplicity we first assume that the sliding triangle of $QP$ is also free of any arcs. 

We start by applying a braidoiding move on $QP$ labeled with $u$ at the point $P$ to obtain a pair of braidoid strands. Then we apply an $L_o$-move at a point (the point $P^*$ in the figure) that is arbitrarily close to $P$ so that there is no endpoint or braidoid end in the vertical strip defined by $P$ and $P^*$. 
Let $Q^*$ be the point where the vertical line passing at $P$ intersects the resulting lower piece of the sub-strand containing $Q$. We assume that the piece of strand containing $Q^*$ slopes downward. Deletion of an $L_u$-move at the point $Q^*$ cancels the pair of braidoid strands labeled $u$. By the same figure we can verify that the resulting labeled braidoid diagram can be turned into to the labeled braidoid diagram obtained by a braidoiding move on $QP$ when it is labeled with $o$, by a sequence of $L$-moves: First by applying an $L_o$-move at $Q^*$ and then deleting an $L_o$-move to cancel one of the  pairs of braidoid strands.

The sliding triangle of the up-arc may intersect with other arcs of the diagram as illustrated in Figure \ref{fig:relabel}. In this figure, we show a case where a free up-arc labeled with $u$ intersects the sliding triangle of the free up-arc labeled with $o$. By using Lemma \ref{lem:subd}, we first subdivide the free up-arc labeled with $o$ into small enough sub-arcs to make all disks of the corresponding sliding triangles free of arcs.  We now re-label each of the free up-arcs resulting from the new subdivision with $o$. By the discussion above, we know that the diagram obtained by applying braidoiding moves at each added subdividing points on the up-arc is $L$-equivalent to the one obtained by applying braidoiding moves to the up-arc in the third row where each sub-arc is labeled with $u$. Finally, again by Lemma \ref{lem:subd} we know that if we delete all the added subdividing points on the up-arc keeping the initial subdividing points and the label $u$ then by a braidoiding move we obtain a labeled braidoid diagram that is $L$-equivalent to the initial labeled braidoid diagram obtained by a braidoiding move applied on the up-arc labeled with $u$. Notice that, to ensure the classical triangle condition in the final step we also change the labeling of the intersecting arc from $u$ to $o$.
 \end{proof}

\begin{figure}[H]
\centering
\includegraphics[width=.9\columnwidth]{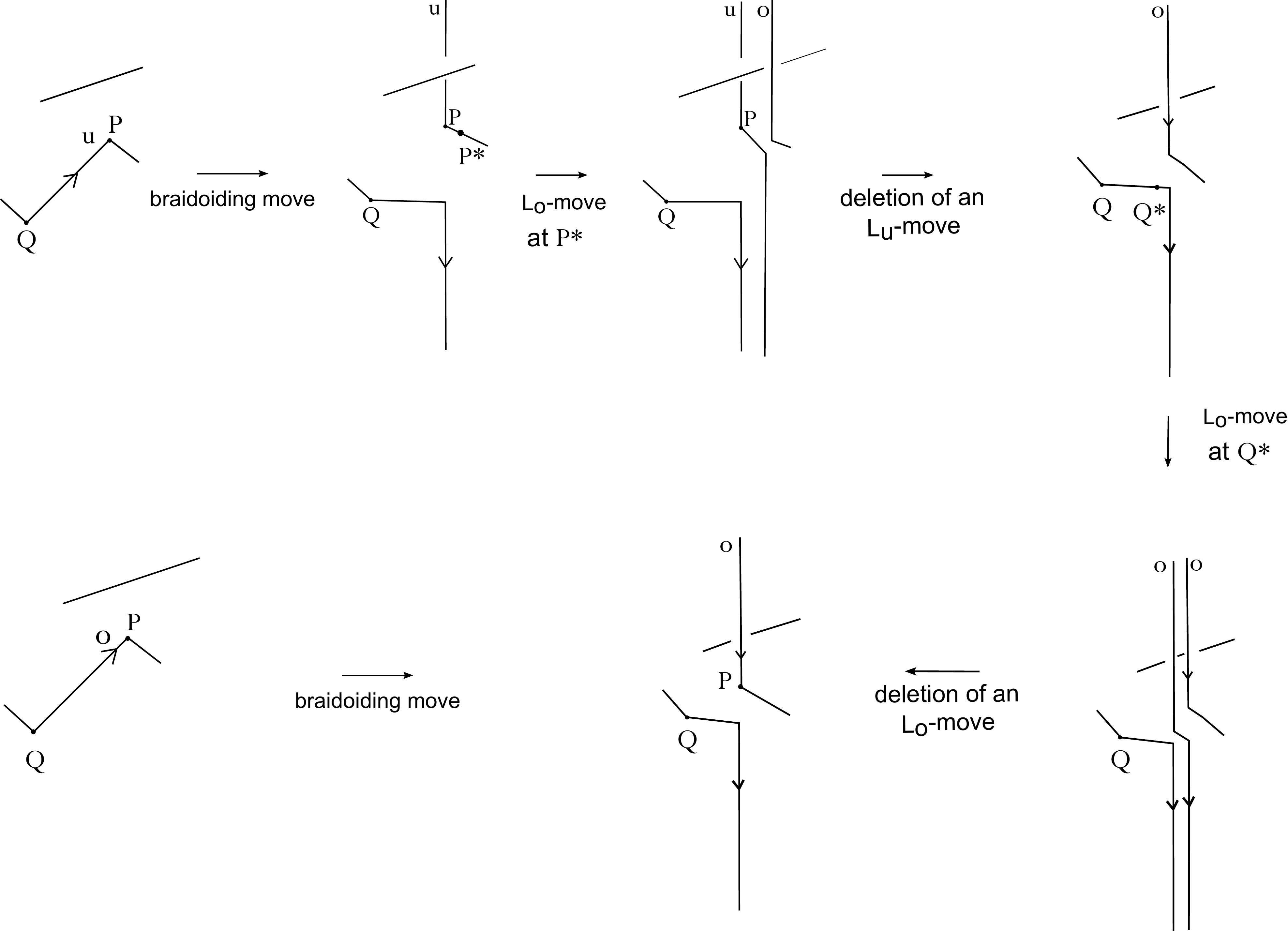}
\caption{Re-labeling the free up-arc and the $L$-equivalence}
\label{fig:free}
\end{figure}

\begin{figure}[H]
\centering\includegraphics[width=1\columnwidth]{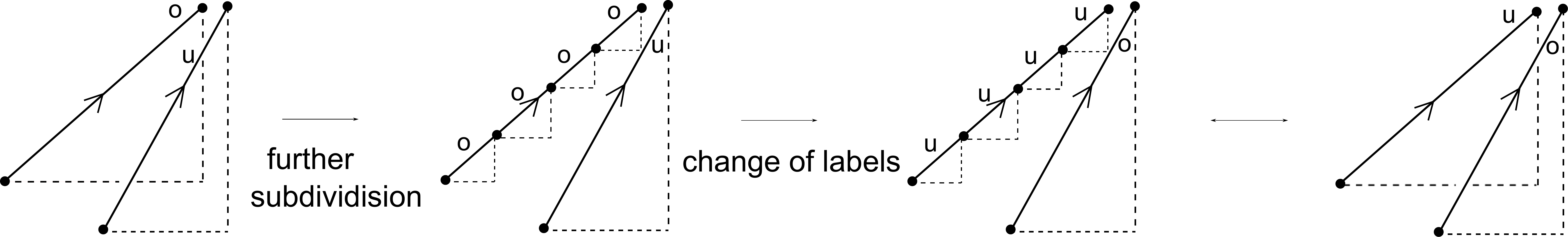}
\caption{Re-labeling the free up-arcs}
\label{fig:relabel}
\end{figure}
\begin{corollary}
If we have an appropriate choice of relabelling the up-arcs resulting from a further subdivision on $K$ then by Lemmas \ref{lem:subd} and \ref{lem:free}, the resulting labeled braidoid diagrams are $L$-equivalent. 
\end{corollary}

\begin{lemma}\label{lem:freesub}
 Two labeled braidoid diagrams that are obtained with respect to any two subdivisions $S_1, S_2$ on a knotoid diagram $K$ that satisfy the triangle conditions, are $L$-equivalent to each other.
\end{lemma}
\begin{proof}
It is clear that any subdivision further than $S_1$ or $S_2$ satisfies both triangle conditions. Consider the further subdivision $S_1\cup S_2$ of both $S_1$ and $S_2$  on $K$. By Lemmas \ref{lem:subd} and \ref{lem:free},  the labeled braidoid diagram that results from the subdivision $S_1\cup S_2$ is $L$-equivalent to the labeled braidoid diagrams that result from the subdivision $S_1$ and $S_2$. Since the $L$-equivalence is an equivalence relation, the lemma follows.
\end{proof}

Finally we shall show the $L$-equivalence under the $\Omega$-moves.
\begin{lemma}\label{lem:delta}
Two (multi-)knotoid diagrams in $\mathbb{R}^2$ are related to each other by planar isotopy and $\Omega$- moves only if the corresponding labeled braidoid diagrams are related to each other by labeled braidoid isotopy moves,  $L$-moves, fake swing moves or fake forbidden moves.
\end{lemma}
\begin{proof}
Let us start with the observation that by imposing the classical triangle condition, we can assume that the knotoid isotopy moves take place without interfering with the braidoiding process of the rest of the diagrams that lie outside the local regions of the moves. In fact we can assume the up-arcs outside the move disks are all turned into braidoid strands and only the arcs lying in the move disks are left for elimination.

Examining the $\Omega_1$, $\Omega_2$ and $\Omega_3$-moves follows similarly with the examination of Reidemeister moves on classical knots/links under the braiding moves \cite{LR1, Lathesis}. Here in Figure \ref{fig:reidpf} we give an illustration for the $\Omega_1$-move away from the endpoints and how it is transformed to an $L$-move under a braidoiding move. 
\begin{figure}[H]
\centering 
\includegraphics[width=.8\textwidth]{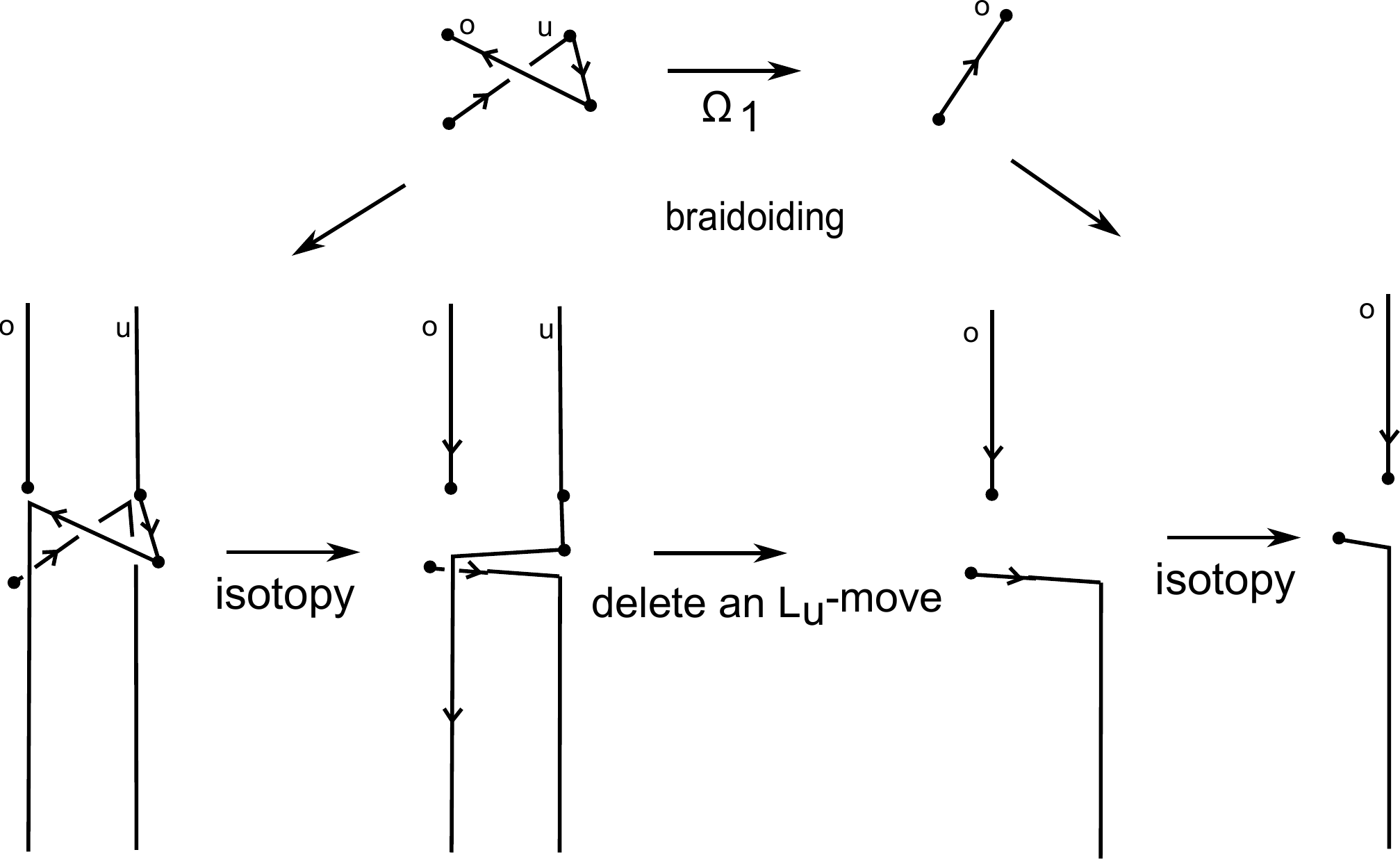}
\caption{An $\Omega_1$-move under braidoiding }
\label{fig:reidpf}
\end{figure}

In this setting, it is crucial to examine specifically the swing moves displacing the endpoints of a knotoid/multi-knotoid diagram.  A swing move may displace an endpoint that lies on a down-arc and in a way that the endpoint does not change its position with respect to a vertical line passing through a cut-point in the diagram then the resulting (labeled) braidoid diagrams are clearly related to each other by a restricted swing move. 

A swing move may move an endpoint so that the swinging arc changes from being a down-arc to being an up-arc by keeping the endpoint at the same vertical alignment as in Figure \ref{fig:case1}. In the braidoiding process, the down-arc is kept but a braidoiding move is applied on the resulting up-arc. It can be verified by the figure that the resulting labeled braidoid diagrams obtained from two knotoid diagrams which are related by such swing move,  are related to each other by an $L$-move. See Figure \ref{fig:case1}. Note that the points $P$ and $P^*$ shown in the figure that are chosen for applying the $L_o$-move on the resulting braidoid diagram and the braidoiding move on the resulting knotoid diagram, respectively, are vertically aligned.

 A swing move may also cause the endpoint to change its position with respect to a vertical line passing through a cut-point on the diagram as illustrated in Figure \ref{fig:case2}. This means that the endpoint of the  swinging arc in the resulting labeled braidoid diagram crosses the vertical line passing through the new pair of corresponding braidoid ends. This situation can be examined in two separate cases. One case is that the swinging arc is a down-arc and remains as a down-arc during the move. In this case, no braidoiding move applies on the arc and the resulting braidoid diagrams are related to each other by a fake forbidden move (recall Definition \ref{defn:fake}). The other case is that the swinging arc is an up-arc and the endpoint is displaced with respect to a vertical line passing through a cut-point lying on some up-arc which can be the swinging arc itself. In this case, we apply a last braidoiding move on the swinging arc before and after the move. Then the comparison of the resulting (labeled) braidoid diagrams is reduced to the previous case and they are related to each other by a fake swing move when the cut-point lies on the swinging arc or a fake forbidden move as discussed in the previous case. 
 
 From the above discussion we can deduce Lemma \ref{lem:delta}.
 \end{proof}

\begin{figure}[H]
\centering\includegraphics[width=.6\textwidth]{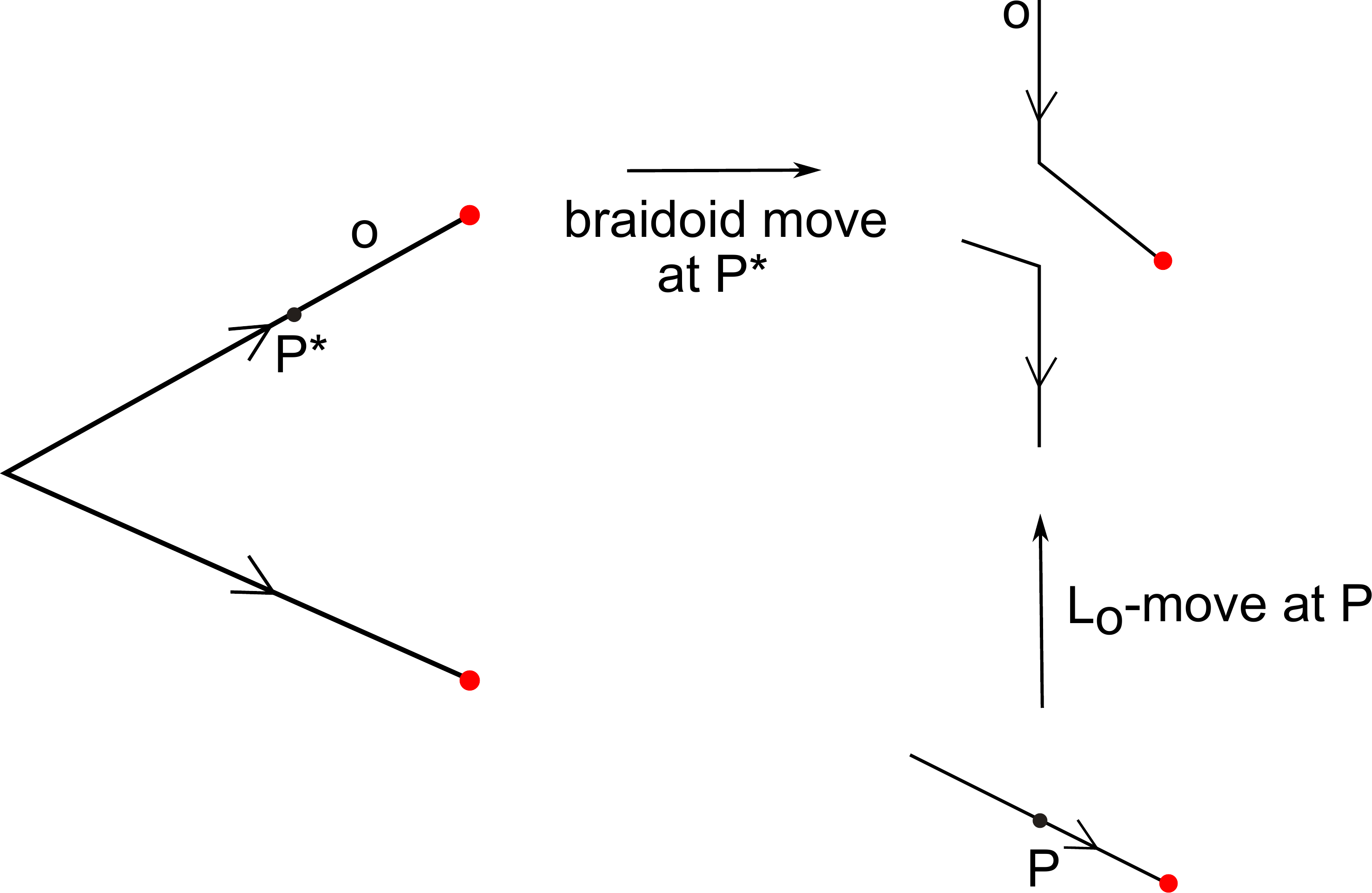}
\caption{ A swing move changing a down-arc to an up-arc}
\label{fig:case1}
\end{figure}
\begin{figure}[H]
\centering\includegraphics[width=.7\textwidth]{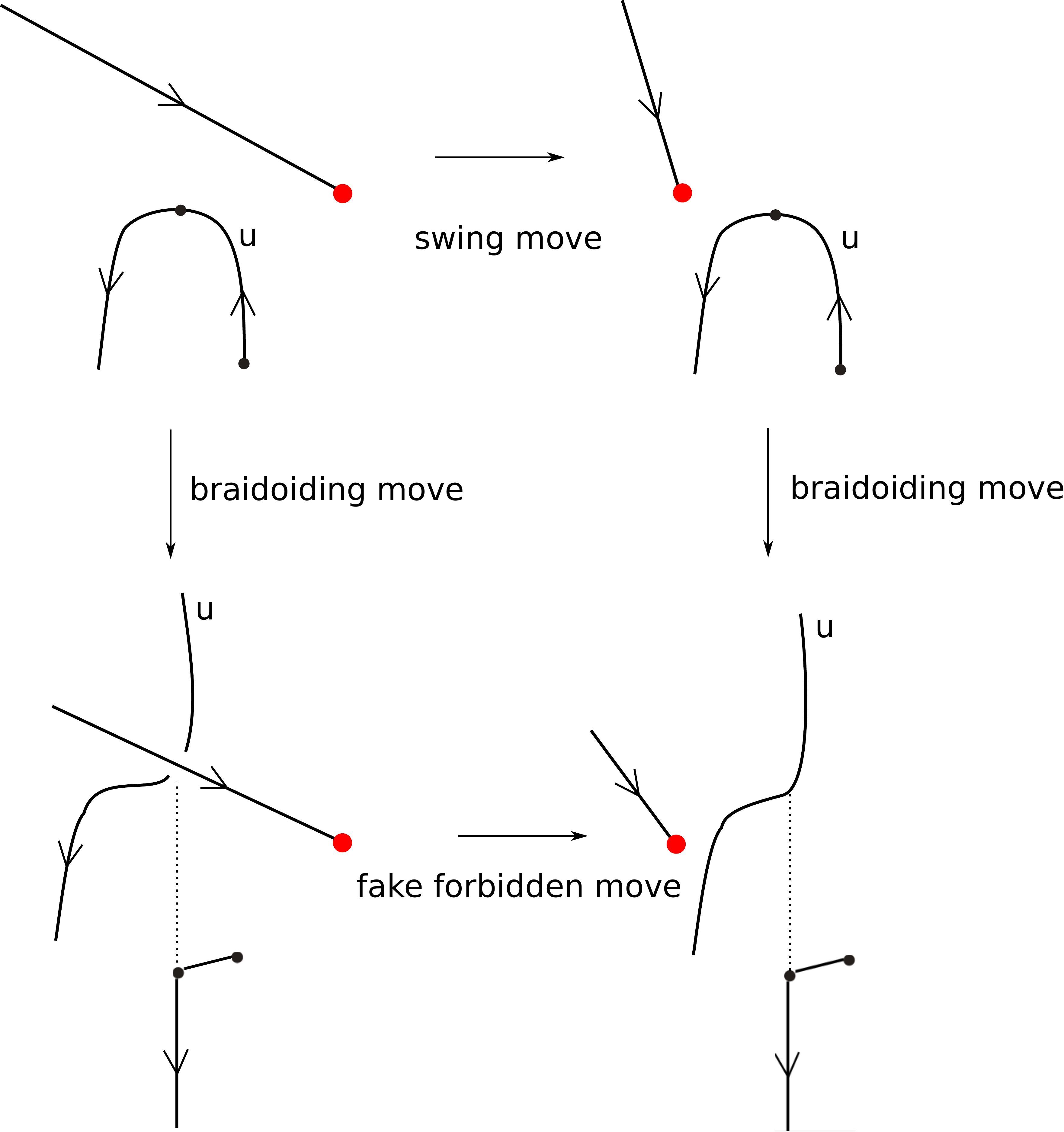}
\caption{ A swing move displacing an endpoint with respect to a cut-point}
\label{fig:case2}
\end{figure}

\begin{lemma}\label{lem:forb}
A fake forbidden move on a labeled braidoid diagram is composed of a finite number of $L$-moves and fake or restricted swing moves.
\end{lemma}
\begin{proof}
The proof can be verified by Figure \ref{fig:forbcomp}. We see  in the figure that a fake forbidden move is decomposed into a sequence of $L_{u}$-moves and a restricted swing move. This can be varied to include $L_{o}$-moves (when the swinging arc goes under a braidoid strand) and fake swing moves (when the swinging arc swings across the vertical line determined by the braidoid end that is connected to the swinging endpoint with the swinging arc, see the first instance of Figure \ref{fig:fakee}.
\begin{figure}[H]
\centering\includegraphics[width=1\textwidth]{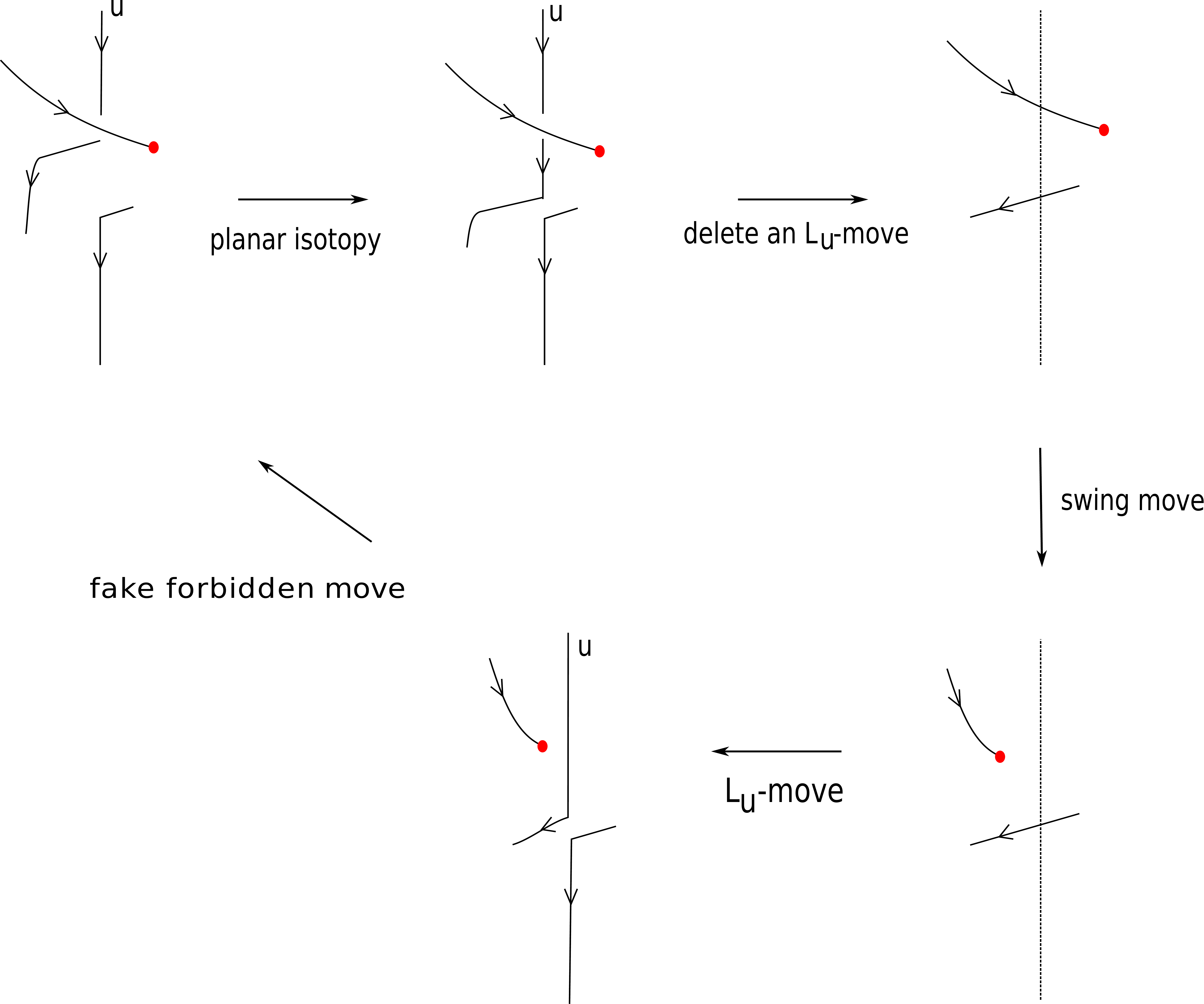}
\caption{ A fake forbidden move is decomposed into $L$-moves and swing moves}
\label{fig:forbcomp}
\end{figure}

\end{proof}

\begin{proposition}\label{prop:del}
Two (multi-)knotoid diagrams in $\mathbb{R}^2$ are related to each other by planar isotopy and $\Omega$-moves only if the corresponding labeled braidoid diagrams are $L$-equivalent. \end{proposition}
\begin{proof}
The proposition follows from Lemma \ref{lem:delta} and Lemma \ref{lem:forb}.
\end{proof}
From Lemma \ref{lem:freesub} and Proposition \ref{prop:del} it follows that the braidoing map $br$ is a well-defined map. Lastly it needs to be shown that $br$ is the inverse map of $cl_L$.
Let $B$ be a labeled braidoid diagram. It is clear that the closure diagram of $B$, $\widehat{B}$, is a knotoid diagram in general position whose only up-arcs are the connection arcs taking place in closing $B$. The braidoiding algorithm eliminates each up-arc of $\widehat{B}$ and turns it into a labeled braidoid diagram which is isotopic to $B$. Thus,
$$br \circ cl_L = id.$$
 Given a knotoid diagram $K$ in general position, by applying the braidoiding algorithm we obtain a labeled braidoid diagram $B$. Clearly the closure of any labeled braidoid diagram that is $L$-equivalent to $B$ is isotopic to $K$. 
 Hence it follows\\
$$cl_L \circ br=id.$$

By the above the proof of Theorem \ref{thm:markov} is now completed.
\end{proof}

\section{\textbf{From braidoids to braids}}\label{sec:tobraids}

In this section we show another application of the $L$-moves, namely the induced relation between the set of all braidoid diagrams and the set of all classical or virtual braid diagrams. 
 
We define a {\it virtual braidoid diagram} to be a braidoid containing classical as well as virtual crossings and a {\it virtual braidoid} to be the equivalence class of virtual braidoid diagrams under the moves of the braidoid diagrams extended by the virtual braid moves. For more on virtual braids and the virtual braid group the reader may consult \cite{Kau2,KL1,Ba} and references therein. Further, for  analogues of the Alexander and the Markov theorems for virtual knots and virtual braids see \cite{KL2,Ka}. 

 In analogy to knotoids closing to classical or virtual knots, a braidoid resp. a virtual braidoid diagram is closed to a classical resp. a virtual braid diagram by connecting its endpoints with a simple arc in the plane. There are only finitely many intersection points between the connecting arc and the braidoid diagram,  which are  transversal double points endowed with under, over resp. virtual crossing data. By a classical topological argument, the arc connecting the endpoints of a braidoid diagram is unique up to isotopy in the classical or in the virtual sense.

\subsection{The underpass closure and the virtual closure}


  We fix the connecting arc to be passing {\it under} any other arc. In order to receive as outcome a braid diagram we must also ensure the braid monotonicity condition. There can be two cases for the connection. If the head of the braidoid diagram appears before its leg (as $t$ increases) then a downward directed arc can be chosen to connect them, see the top row of Figure~\ref{fig:u} for an abstract illustration. If, however, the leg of the braidoid diagram appears before its head then any arc connecting them is an arc that is directed upward, an up-arc, in the resulting tangle diagram. In this case, we apply a braiding move 
	on the connecting arc which turns it into two new corresponding braid strands.  See the bottom row of Figure~\ref{fig:u} for an abstraction of this case. 
  \begin{definition}\normalfont
For any given braidoid diagram we define the \textit{underpass closure} of its endpoints as follows. If the head is in a relatively higher position than the leg, we join them together with a straight under-passing arc. If the leg is located higher than the head then we extend both endpoints by a pair of corresponding underpassing braid strands, emanating from the leg and the head respectively, and vertically aligned with a point on the straight arc connecting the two endpoints.
\end{definition}
	
	The above apply analogously for the case where the connecting arc is passing {\it virtually} any other arc. In this case  we obtain in the end a virtual braid. and call the closure the \textit{virtual closure}.  Especially in the case where the leg of the braidoid diagram appears before its head we eliminate the joining up-arc by a virtual braiding move, whereby all crossings in the resulting new pair of strands are virtual, see \cite{KL2}.




\begin{figure}[H]
\centering 
\includegraphics[width=.8\textwidth]{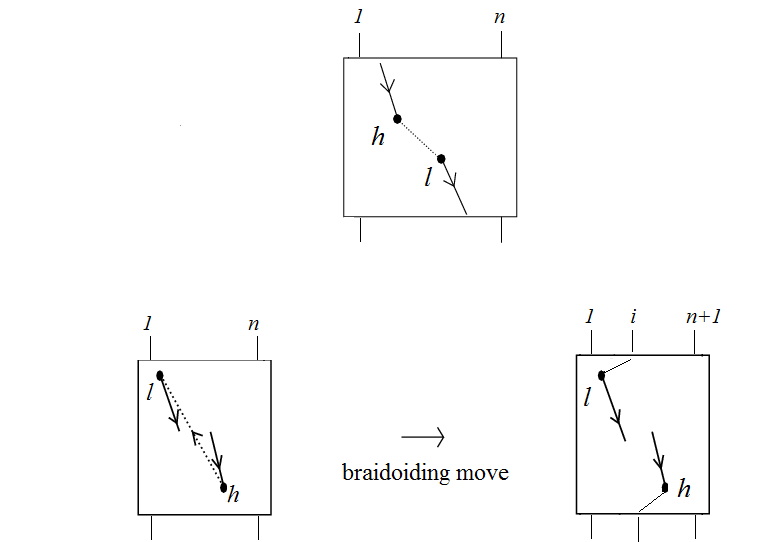}
\caption{Abstract examples for the underpass closure}
\label{fig:u}
\end{figure} 


 \subsection{The induced mappings}

We shall now establish that the underpass closure (resp. the virtual closure) defined on braidoid diagrams induces a well-defined map on the set of braidoids (that is, isotopy classes of braidoid diagrams). We have the following proposition.

\begin{proposition}\label{prop:underp}
The underpass closure  is a well-defined surjective map from the set of braidoids to the set of $L$-equivalence classes of classical braids. Similarly, the virtual closure  is a well-defined surjective map from the set of virtual braidoids to the set of virtual $L$-equivalence classes of virtual braids.
\end{proposition} 
\begin{proof}
It is straightforward to show that the underpass closure is a surjective map. Indeed, by cutting out from a given braid diagram an open arc that either contains no crossings or it is an underpassing arc (resp. a virtual arc), we obtain a braidoid diagram whose underpass closure (resp. virtual closure) is clearly the original braidoid (resp. virtual braidoid) diagram.

Consider a braidoid isotopy on a given braidoid diagram (resp. virtual braidoid isotopy on a virtual braidoid diagram). The isotopy may keep the endpoints fixed or it may change their relative vertical or horizontal positions. If the  isotopy keeps the endpoints fixed then  after the underpass (resp. virtual) closure this isotopy clearly transforms into braid isotopy (recall Section \ref{sec:isotopy}). Suppose that a swing move on one of the endpoints takes place. If the connecting arc is a down-arc then the corresponding braid diagrams clearly differ by braid isotopy (resp. virtual braid isotopy). If, however, the connecting arc is an up-arc then the swing move may change the vertical level (the ordering amongst the strands) of the resulting two braid strands. For an illustration see Figure \ref{fig:sw}. Yet, the two resulting braids will differ by conjugation, which is known to be a special case of $L$-equivalence \cite{LR1} (resp. virtual $L$-equivalence \cite{KL2}).  

Suppose now that a vertical move on the endpoints takes place. If this move does not change the relative heights of the endpoints, it is clear that the connecting arcs are isotopic.  
If, however, the move changes the relative heights of the endpoints then the connecting arc changes from a down-arc to an up-arc or vice versa, see Figure \ref{fig:well}. In this case, the two resulting braid diagrams differ by an $L_under$ move (resp. virtual $L$-move)  applied on the connecting down-arc of the one diagram and at the same vertical level as that of the braiding move applied on the connecting up-arc of the other diagram.

\begin{figure}[H]
\includegraphics[width=.55\textwidth]{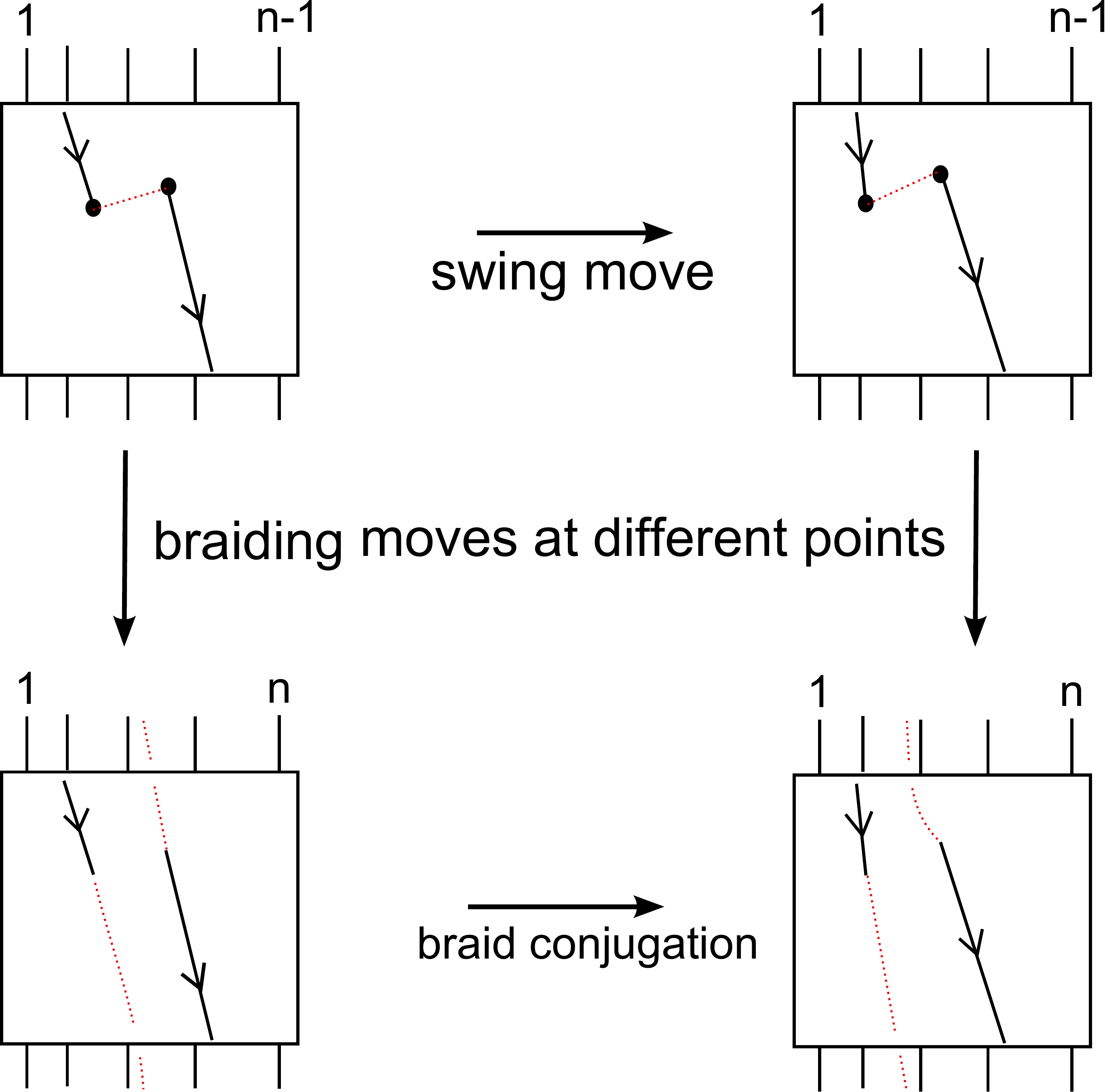}
\caption{A swing move transforms into an $L$-move on the underpass closure}
\label{fig:sw}
\end{figure}

\begin{figure}[H]
\includegraphics[width=.55\textwidth]{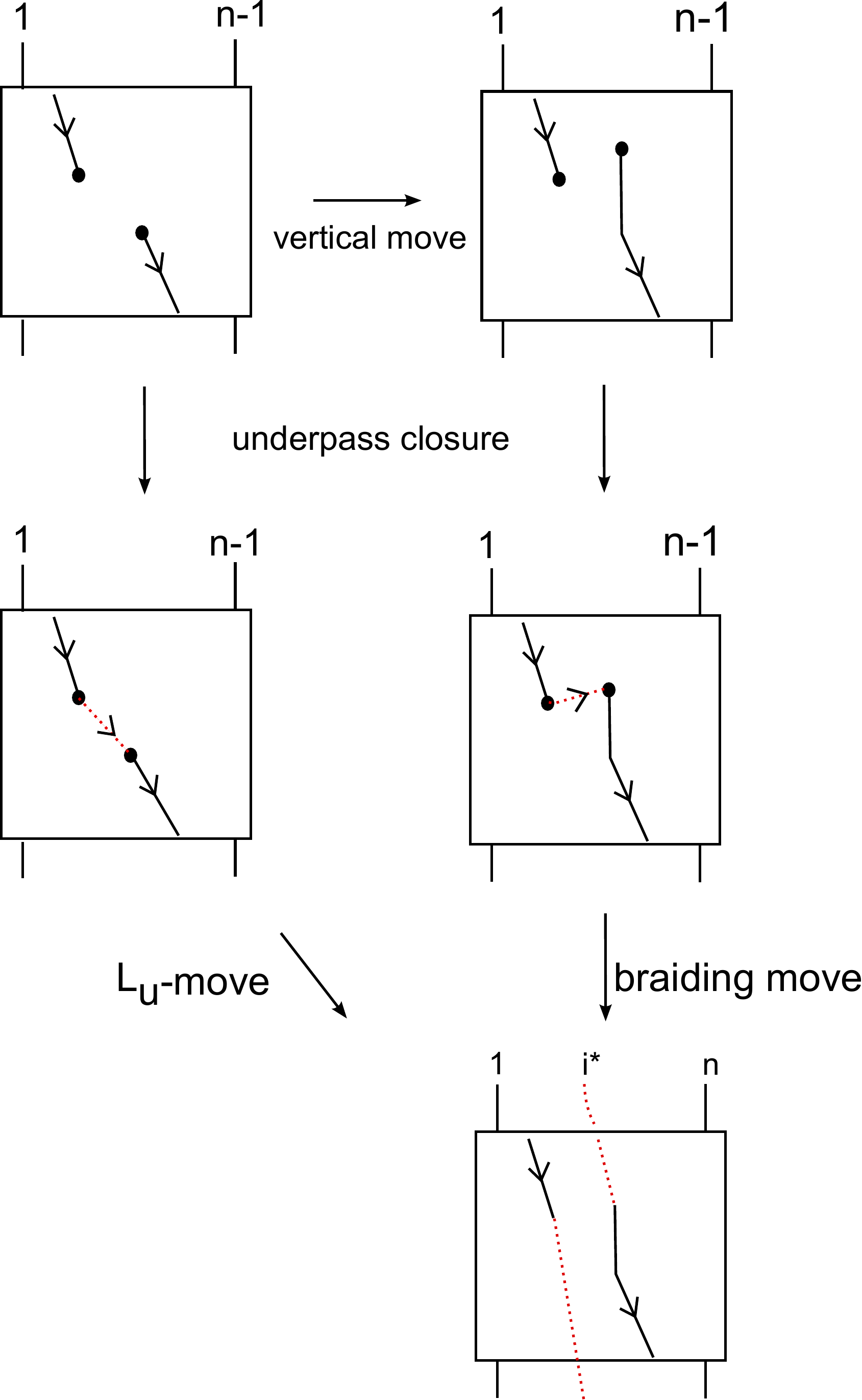}
\caption{A vertical move transforms into an $L$-move on the underpass closure}
\label{fig:well}
\end{figure}

By the arguments above the proof of Proposition \ref{prop:underp} is completed.
\end{proof}

\section{Discussion}
The theory of braidoids is a new diagrammatic setting extending the classical braid theory and giving rise to many new proble ms;   the underlying algebraic structure for braidoids is not known yet and we will examine in a subsequent paper. However, we know how to partition a braidoid diagram into smaller diagrams that we call elementary blocks \cite{GL1, Gthesis}. Knots and knotoids have been used for topological tabulation of proteins which are open molecular chains. In \cite{GL1} we suggest that braidoids can set up an algebraic way for tabulating proteins by the use of elementary blocks. 
\\
\section*{Acknowledgement}
The first author cordially thanks the Oberwolfach Research Institute for Mathematics for providing her a peaceful research environment during the time of completion of this paper.

\end{document}